\documentclass[12pt]{amsart}
\usepackage{url}
\usepackage{amsmath,amsxtra,amssymb,latexsym,epsfig,amscd,amsthm,fancybox,epsfig}
\usepackage[mathscr]{eucal}
\usepackage{graphicx}
\usepackage{multicol,xcolor}
\usepackage{epsfig} 
\usepackage{epstopdf}
\usepackage{cases}
\usepackage{subfig}
\usepackage{color}
\usepackage{hyperref}
\usepackage{bigints}
\usepackage{enumitem}

\usepackage{natbib}
\usepackage{hyperref}
\hypersetup{
    colorlinks=true,
    linkcolor=blue,
    filecolor=magenta,
    urlcolor=cyan,
    citecolor=blue,
}

\usepackage{thmtools,thm-restate}

\newtheorem{thm}{Theorem}[section]
\newtheorem {asp}{Assumption}[section]
\newtheorem{lm}{Lemma}[section]
\newtheorem{rmk}{Remark}[section]

\newtheorem{deff}{Definition}[section]

\newtheorem{prop}{Proposition}[section]
\theoremstyle{definition}

\theoremstyle{remark}

\numberwithin{equation}{section}

\DeclareMathOperator{\suppo}{supp}
\DeclareMathOperator{\Conv}{Conv}
\newcommand{\eps}{\varepsilon}

\newcommand{\M}{\mathcal{M}}
\newcommand{\F}{\mathcal{F}}

\newcommand{\E}{\mathbb{E}}
\newcommand{\BE}{\mathbf{E}}
\newcommand{\BB}{\mathbf{B}}

\newcommand{\BX}{\mathbf{X}}
\newcommand{\bx}{\mathbf{x}}

\newcommand{\by}{\mathbf{y}}

\newcommand{\bc}{\mathbf{c}}

\newcommand{\bp}{\mathbf{p}}

\newcommand{\N}{\mathbb{N}}

\newcommand{\PP}{\mathbb{P}}

\newcommand{\K}{\mathcal{K}}

\newcommand{\R}{\mathbb{R}}

\newcommand{\Lom}{\mathcal{L}}
\newcommand{\U}{\mathcal{U}}

\newcommand{\wtd}{\widetilde}
\numberwithin{equation}{section}

\newcommand{\Rx}{\mathbb{R}_{1+}^\circ}

\newcommand{\Ri}{\mathbb{R}_{i+}^\circ}

\newcommand{\bed}{\begin{displaymath}}
\newcommand{\eed}{\end{displaymath}}
\newcommand{\bea}{\bed\begin{array}{rl}}
\newcommand{\eea}{\end{array}\eed}

\newcommand{\barray}{\begin{array}{ll}}
\newcommand{\earray}{\end{array}}
\newcommand{\diag}{{\rm diag}}

\newcommand{\1}{\boldsymbol{1}}
\newcommand{\0}{\boldsymbol{0}}
\newcommand{\bdelta}{\boldsymbol{\delta}}

\newcommand{\dist}{\mathrm{dist}}
\definecolor{OliveGreen}{cmyk}{0.64,0,0.95,0.40}

\def\bar{\overline}
\def\hat{\widehat}
\def\a.s{\text{\;a.s.\;}}

\def\dist{{\rm dist}}

\title[Three-dimensional stochastic ecological systems]{A classification of the dynamics of three-dimensional stochastic ecological systems }
\author[A. Hening]{Alexandru Hening }
\address{Department of Mathematics\\
Texas A\&M University\\
Mailstop 3368\\
College Station, TX 77843-3368\\
United States
}
\address{Department of Mathematics\\
Tufts University\\
Bromfield-Pearson Hall\\
503 Boston Avenue\\
Medford, MA 02155\\
United States
}
\email{al.hening@gmail.com}

\author[D.H. Nguyen]{Dang H. Nguyen }
\address{Department of Mathematics \\
University of Alabama\\
345 Gordon Palmer Hall\\
Box 870350 \\
Tuscaloosa, AL 35487-0350 \\
United States}
\email{dangnh.maths@gmail.com}
\author[S.J. Schreiber]{Sebastian J. Schreiber}
\address{Department of Evolution and Ecology \\
 University of California, Davis\\
  One Shields Avenue\\
  Davis, CA 95616\\
 United States}
  \email{sschreiber@ucdavis.edu}


\begin{document}

\begin{abstract}
The classification of the long-term behavior of dynamical systems is a fundamental problem in mathematics. For both deterministic and stochastic dynamics specific classes of models verify Palis' conjecture: the long-term behavior is determined by a finite number of stationary distributions. In this paper we consider the classification problem for stochastic models of interacting species.  For a large class of three-species, stochastic differential equation models, we prove a variant of Palis' conjecture: the long-term statistical behavior is determined by a finite number of stationary distributions and, generically, three general types of behavior are possible: 1) convergence to a unique stationary distribution that supports all species, 2) convergence to one of a finite number of stationary distributions supporting two or fewer species, 3) convergence to convex combinations of single species, stationary distributions due to a rock-paper-scissors type of dynamic. Moreover, we prove that the classification reduces to computing Lyapunov exponents (external Lyapunov exponents) that correspond to the average per-capita growth rate of species when rare. Our results stand in contrast to the deterministic setting where the classification is incomplete even for three-dimensional, competitive Lotka--Volterra systems. For these SDE models, our results also provide a rigorous foundation for ecology's modern coexistence theory (MCT) which assumes the external Lyapunov exponents determine long-term ecological outcomes.
\end{abstract}
\maketitle


\section{Introduction}\label{s:intro}

Since the time of Newton and Bernoulli~\citep{newton1687,bernouilli1738}, dynamical models, whether they be deterministic or stochastic, have been used  to describe how physical, economic, and biological systems change over time. A fundamental challenge for these models has been and continues to be a classification of their long-term behaviors. For finite-state Markov chains, this long-term statistical behavior is characterized by a finite number of stationary distributions~\citep{norris1998}. For deterministic models, such as ordinary differential equations, \citet{palis2005,palis2008} conjectured that typically there are a finite number of stationary distributions characterizing the long-term statistical behavior for most initial states of the model. Decades of work have identified several classes of deterministic models, including Axiom A systems~\citep{young1986}, one-dimensional maps~\citep{kozlovski2003}, and partially hyperbolic systems~\citep{alves_araujo2007}, for which Palis' conjecture holds. However, for general, three-dimensional deterministic models, this conjecture still remains unproven. Here, we consider this type of classification problem for stochastic models of interacting populations. For these systems in three dimensions, we prove that, generically, there are three types of long-term statistical behavior that are characterized by a finite number of stationary distributions. This classification  is determined by certain Lyapunov exponents that correspond to the average per-capita growth rate of rare species. We conjecture that this classification scheme also  holds for higher dimensions.

For dynamical models in ecology, evolution, and epidemiology, the state variables may represent the densities of interacting species of plants, animals, microbes, and viruses. For these models, two fundamental problems of scientific and practical interest are identifying which of the species persist and which go extinct, and  understanding  the long-term statistical behavior of the densities of the persisting species~\citep{elith_leathwick2009,thieme2018,ellner2019}. There is a large theoretical literature devoted to the study of persistence and extinction for deterministic models. The most famous are studies of two competing species due to  Lotka and Volterra. Under the assumption of mass action interactions, \citet{V28} showed that, generically, one species drives the other species extinct when the species are competing for a single limiting resource; a prediction with extensive empirical support~\citep[see, e.g, the review by][]{wilson2007}. Alternatively, \cite{L25} demonstrated under what conditions competing species could coexist, setting the stage for modern coexistence theory~\citep{C00,ellner2019}. There has been a significant amount of work dedicated to the classification of the long term behavior of deterministic Lotka--Volterra systems \citep{B83, B95, Z93, HS94, T96,HS98}. While there is a full classification in dimension two \citep{B83, B95}, the classification is still incomplete for three dimensions even in the special case of competitive systems \citep{Z93, Z98, HS94,schreiber1999, XL00, GYW06, GY09}.

While theoretical population biologists have discovered many important phenomena by studying these deterministic models, population dynamics in nature are often buffeted by stochastic fluctuations in environmental factors.  As a result, one has to study the interaction between the population dynamics and these random environmental fluctuations to determine conditions for persistence and extinction. One successful approach to this problem has been the use of stochastic difference equations for discrete-time \citep{C82, CE89, C00,  BS09, S12, BS18, H19, HNC20} and stochastic differential equations (SDE) for continuous-time \citep{ERSS13, EHS15, LES03,SBA11, BHS08, HNY16, HN16, HN17b, HN17, B18, HL20, HNC20}.

For two dimensional SDEs, \citet{HN16} showed that, generically, the dynamics can be classified into four types: (i) both populations go asymptotically extinct with probability one, (ii) one population goes extinct while the other approaches a unique, positive stationary distribution with probability one, (iii) either species goes extinct with complementary positive probabilities, while the other approaches a unique stationary distribution associated with it, or (iv) both populations persist with probability one and approach a unique, positive stationary distribution. This classification is determined by Lyapunov exponents corresponding to the per-capita growth rates of species when they are infinitesimally rare.

Here, we extend this classification to three-dimensional systems. This extension leads to generalizations of the two-dimensional outcomes (i)--(iii) and introduces a different type of outcome. The generalization of (i)--(iii) is that for any collection of subcommunities ,i.e., subsets of species, where no subcommunity is contained in another, the ecological dynamics converge to a stationary distribution associated with any one of these subsets with positive probability. Alternatively, the new dynamic is a rock-paper-scissor extinction dynamic whereby the long-term statistical behavior is governed by convex combinations of three single species, stationary distributions. For SDEs of Lotka-Volterra type, we show that the classification reduces to solving a finite number of systems of linear equations. We also illustrate how conditions for species coexistence for the stochastic models can differ substantially from the coexistence conditions for the corresponding deterministic models. We conclude by summarizing our main results and making a conjecture of how to classify these systems in higher dimensions. We also discuss the implications for ecology's modern coexistence theory~\citep{C00,ellner2019}.

The paper is structured as follows. In Section \ref{s:models} we describe the models and our assumptions. The main results appear in Section \ref{s:main}. The proofs of the various propositions and theorems appear in Sections \ref{s:app1}, \ref{s:app2} and \ref{s:rock-paper-scissors} while the case by case classification of the dynamics is in Section \ref{s:results}. We apply our results to Lotka-Volterra systems in Section \ref{s:LV}. We conclude the paper with a discussion in Section \ref{s:disc}.

\section{Models and Assumptions}\label{s:models}
We consider the dynamics of $n\le 3$ interacting species whose densities at time $t$ are given by $\BX(t)=(X_1(t),X_2(t),\dots,X_n(t))$. To capture the effects on environmental stochasticity, the species dynamics are modeled by a system of stochastic differential equations of the form
\begin{equation}\label{e:system}
dX_i(t)=X_i(t) f_i(\BX(t))dt+X_i(t)g_i(\BX(t))dE_i(t), ~i=1,\dots,n
\end{equation}
where $\BE(t)=(E_1(t),E_2(t), \dots E_n(t))^T=\Gamma^\top\BB(t)$, $\Gamma$ is a $n\times n$ matrix such that $\Gamma^\top\Gamma=\Sigma=(\sigma_{ij})_{n\times n}$
and $\BB(t)=(B_1(t),B_2(t),\dots, B_n(t))$ is a vector of independent standard Brownian motions adapted to the filtration $\{\F_t\}_{t\geq 0}$. The system \eqref{e:system} is called a Kolmogorov system or generalized Lotka-Volterra system. The functions $f_i(\BX)$ correspond to the per-capita growth rate of species $i$ and the functions $g_i(\BX)$ determine the per-capita magnitude of the environmental fluctuations experienced by species $i$. Namely, $\rm{Var}[X_i(t+\Delta t)-X_i(t)|\BX(t)=\BX]=(X_ig_i(\BX))^2 \sigma_{ii} \Delta t + o(\Delta t)$. We refer the reader to the work by \cite{T77,G84, SBA11, HN16} for more details about why \eqref{e:system} makes sense biologically. We will denote by $\PP_\by(\cdot)=\PP(~\cdot~|~\BX(0)=\by)$ and $\E_\by[\cdot]=\E[~\cdot~|~\BX(0)=\by]$ the probability and expected value given that the process starts at $\BX(0)=\by\in \R_+^{n}:=[0,\infty)^n$. We will define the interior of the positive orthant by $\R_+^{n,\circ}:=(0,\infty)^n$.

To ensure the dynamics of \eqref{e:system} are well-defined and are stochastically bounded, we make the following standing assumptions.
\begin{asp}\label{a.nonde} The following hold:
\begin{enumerate}
\item $\diag(g_1(\bx), \dots, g_n(\bx))\Gamma^\top\Gamma\diag(g_1(\bx),\dots, g_n(\bx))=(g_i(\bx)g_j(\bx)\sigma_{ij})_{n\times n}$
is a positive definite matrix for any $\bx\in \R^n_+:=[0,\infty)^n$.
\item $f_i(\cdot), g_i(\cdot):\R^n_+\to\R$ are locally Lipschitz functions for any $i=1,\dots,n.$
\item
There are $\bc=(c_1,\dots,c_n)\in\R^{n,\circ}_+:=(0,\infty)^n$, $\gamma_b>0$ such that
\begin{equation}\label{a.tight}
\limsup\limits_{\|\bx\|\to\infty}\left[\dfrac{\sum_{i=1}^n c_ix_if_i(\bx)}{1+\sum_{i=1}^n c_ix_i}-\dfrac12\dfrac{\sum_{i,j=1}^n \sigma_{ij}c_ic_jx_ix_jg_i(\bx)g_j(\bx)}{(1+\sum_{i=1}^n c_ix_i)^2}+\gamma_b\left(1+\sum_{i=1}^n (|f_i(\bx)|+g_i^2(\bx))\right)\right]<0.
\end{equation}
\end{enumerate}
\end{asp}
\begin{rmk}
Part (1) of Assumption \ref{a.nonde} to ensure that the solution to \eqref{e:system} is a non-degenerate diffusion. Parts (2) and (3) guarantee the existence and uniqueness of strong solutions to \eqref{e:system}. Moreover, (3) implies the tightness of the family of transition probabilities of the solution to \eqref{e:system}. Note that equation \eqref{a.tight} is satisfied in most ecological models as long as  intraspecific competition is sufficiently strong.
\end{rmk}
\begin{asp}\label{a.extn2}
Suppose that there is $\delta_1>0$ such that
$$
\lim\limits_{\|\bx\|\to\infty} \dfrac{\|\bx\|^{\delta_1}\sum_{i=1}^n g_i^2(\bx)}{1+\sum_{i=1}^n(|f_i(\bx)|+|g_i(\bx)|^2)}=0.
$$
\end{asp}

\begin{rmk}
Assumption \ref{a.extn2} forces the growth rates of $g_i^2(\cdot)$ to be slightly lower than those of $|f_i(\cdot)|$.
This is needed in order to suppress the diffusion part so that
we can obtain the tightness of certain occupation measures.
\end{rmk}

\section{Main Results}\label{s:main}
Under assumption \eqref{a.nonde} one can use the proof from Lemma 3.1 in \cite{HN16} to show the following.
\begin{lm}\label{lm2.0}
Suppose Assumption \ref{a.nonde} holds. Then, for any $\bx\in\R^n_+$ there exists a pathwise unique strong solution $(\BX(t))$ to \eqref{e:system}
with initial value $\BX(0)=\bx$.
The solution $(\BX(t))$ with initial value $\bx(0)=\bx\in \R^{I,\circ}_+$ will stay forever in $\R^{I,\circ}_+$ with probability 1. Moreover, $\BX(t)$ is a Feller process on $\R_+^n$.
\end{lm}

One can associate to the Markov process $\BX(t)$
the semigroup $(P_t)_{t\geq 0}$ defined by its action on bounded Borel
measurable functions $h:\R_+^n\to\R$
\[
P_th(\bx)=\E_\bx[h(\BX(t))], t\geq 0, \bx\in \R_+^n.
\]
The operator $P_t$ can be seen to act by duality on
Borel probability measures $\mu$ by $\mu\to \mu P_t$ where $\mu
P_t$ is the probability measure given by
\[
\int_{\R^3_+}h(\bx) (\mu P_t)(d\bx):=\int_{\R^3_+}P_th(\bx) \mu(d\bx)
\]
for all $h\in C_b(\R_+^3)$.
\begin{deff}
A probability measure $\mu$ on $\R_+^3$ is called  \emph{invariant} if
$P_t\mu=\mu$ for all $t\geq 0$. The invariant probability
measure $\mu$ is called  \emph{ergodic} if it cannot be written as a
nontrivial convex combination of invariant probability measures.
\end{deff}

We are interested in understanding the asymptotic, statistical behavior of $\BX$. To this end, we define the normalized random occupation measures
\[
\Pi_t(\cdot):=\dfrac1t\int_0^t\1_{\{\BX(s)\in\cdot\}}ds \mbox{ for all }t>0\]
where $\1_{A}$ is the indicator function which takes the value $1$ on the set $A$ and $0$ on the complement $A^c$.
Denote the weak$^*$-limit set of the family $\left(\Pi_t(\cdot)\right)_{t\geq 1}$ by the random set of probability measures $\U$. These weak$^*$-limit points are almost-surely invariant probability measures for $\BX$ -- see Theorem 9.9 from \cite{EK09} or \cite{HN16}.
For the ergodic invariant probability measures, we make the following definition.

\begin{deff}
For an ergodic invariant probability measure $\mu$ for $\BX$, invariance of the faces of the non-negative cone, meaning that if the process starts in one such subspace then it stays there forever (see Lemma \ref{lm2.0}), implies that there is a unique subset $I\subset\{1,2,\dots,n\}$ such that $\mu(\{\bx\in \R_+^n: x_i>0$ if and only if $i\in I\})=1.$ We define this subset $I$ as the \emph{species support of $\mu$} and denote it as $I_\mu.$ In the special case that $\mu=:\bdelta^*$ is the Dirac measure concentrated at the origin $\0$, $I_\mu=\emptyset$. We denote the set of all ergodic measures by $\M$ and the set of all invariant measures by $\Conv(\M)$.
\end{deff}

For any subset $I\subset \{1,2,3\}$ define
\[
\R_+^{I}=\{(x_1, x_2, x_3)\in\R^3_+: x_i=0\text{ if } i\in I^c\},
\]
$$\R_+^{I,\circ}:=\{(x_1,x_2,x_3)\in\R^3_+: x_i=0\text{ if } i\in I^c\text{ and }x_i>0\text{ if  }x_i\in I\},$$
and $\partial\R_+^{\mu}=\R_+^{I}\setminus \R_+^{I,\circ}$.

Consider any ergodic measure $\mu\in\M$ and assume $\mu\neq \bdelta^*$. Define
$$\R_+^\mu:=\R_+^{I_\mu}=\{(x_1, x_2, x_3)\in\R^3_+: x_i=0\text{ if } i\in I_\mu^c\}.$$
Let
$$\R_+^{\mu,\circ}:=\R_+^{I_\mu,\circ}=\{(x_1,x_2,x_3)\in\R^3_+: x_i=0\text{ if } i\in I_\mu^c\text{ and }x_i>0\text{ if  }x_i\in I_\mu\}$$ and $\partial\R_+^{\mu}:=\R_+^\mu\setminus\R_+^{\mu,\circ}$.
\begin{rmk}
Note that one can show (see \cite{HN16, B18}) that under some natural assumptions the set $\Conv(\M)$ is convex and compact and $\mu$ is ergodic if and only if it cannot be written as a nontrivial convex combinations of invariant probability measures. The ergodic decomposition theorem tells us that any invariant probability measure is a convex combination of ergodic measures. Furthermore, it can be shown that any two ergodic probability measures are either identical or mutually singular and that the topological supports of any mutually singular invariant measures are disjoint. In addition, because the diffusion is non-degenerate and invariant on any subspace $R_+^{I,\circ}$, the topological support of an ergodic measure $\mu$ is  $R_+^{I_\mu,\circ}$. In particular, this implies that $\M$ is finite.
\end{rmk}

For an given initial condition $\by$, we are interested in the probability that an ergodic invariant probability measure $\mu$ characterizes the long-term behavior of $\BX$. With this objective in mind, we make the following definition.

\newcommand{\pmu}{p}
\begin{deff}
Let $\mu$ be an ergodic invariant probability measure for $\BX.$ Define
\[
\pmu_\by(\mu)=\PP_\by\left(\U=\{\mu\} \mbox{ and } \limsup_{t\to\infty}\frac{1}{t}\log X_i(t)<0 \mbox{ for all }i\notin I_\mu \right)
\]
as the probability that the normalized occupation measures converge to $\mu$ and the species not supported by $\mu$ go extinct at an exponential rate. \end{deff}

\begin{rmk} The proofs of our main results also provide upper bounds to $\limsup_{t\to\infty}\frac{1}{t}\log X_i(t)$ almost-surely on the event $\left\{\limsup_{t\to\infty}\frac{1}{t}\log X_i(t)<0\right\}$. \end{rmk}

A case of particular importance is when there is an ergodic invariant probability measure that supports all species and characterizes the long term dynamics for all positive initial conditions. We write $\by\gg 0$ if $y_i>0$ for all $i.$

\begin{deff}
The process $\BX$ is \emph{strongly stochastically persistent} if it has a unique invariant probability measure $\mu$ with $I_\mu=\{1,2,\dots,n\}$ such that $\pmu_\by(\mu)=1$ for all $\by \gg 0$.
\end{deff}

To characterize $\pmu_\by(\cdot)$, we make use of certain Lyapunov exponents associated with the derivative cocycle of \eqref{e:system}. For the directions corresponding to species which are not supported by an ergodic measure, these Lyapunov exponents take on a particularly simple form.

\begin{deff}
For an ergodic probability measure $\mu$ define
\begin{equation}\label{e:lyapunov}
\lambda_i(\mu):=\int_{\R^n_+}\left(f_i(\bx)-\dfrac{\sigma_{ii}g_i^2(\bx)}2\right)\mu(d\bx).
\end{equation}
For $i\notin I_\mu$, $\lambda_i(\mu)$ is an \emph{external Lyapunov exponent}. These external Lyapunov exponents determine the infinitesimal per-capita rate of growth of species not supported by $\mu.$
\end{deff}

For $i\in I_\mu$, the following proposition from \citet{HN16} implies that the average per-capita growth rate of the supported species equals $0$. For these $i$, $\lambda_i(\mu)$ does not correspond to a Lyapunov exponent associated with the derivative cocycle of $\BX$'s dynamics.
\begin{restatable}{prop}{zero}\label{p:zero}
Suppose that Assumptions~\ref{a.nonde}--\ref{a.extn2} hold. If $\mu$ is an ergodic invariant probability measure, then $\lambda_i(\mu)=0$ for all $i\in I_\mu.$
\end{restatable}
The next two propositions describe previous results for $n=1$ and $n=2$ species that follow from \citet{HN16}.\\

\begin{prop}~\label{prop:1d}
Assume $n=1$ and Assumptions~\ref{a.nonde}--\ref{a.extn2} hold.  If $\lambda_1(\bdelta^*)>0$, then $\BX$ is strongly, stochastically persistent. If $\lambda_1(\bdelta^*)<0$, then $\pmu_\by(\bdelta^*)=1$ for all $\by \gg 0$.\\
\end{prop}

Proposition~\ref{prop:1d} highlights that when the external Lyaponov exponent $\lambda_1(\bdelta^*)$ is non-zero, strong conclusions can be drawn about the long-term statistical behavior of \eqref{e:system}. All of our results rely on the following generalization of this assumption. \begin{asp}\label{a:external} For every ergodic invariant probability measure $\mu$, the external Lyapunov exponents are non-zero i.e. $\lambda_i(\mu)\neq 0$ for $i\notin I_\mu$. \end{asp} As we show later, Assumption~\ref{a:external} holds generically for \eqref{e:system} in the sense that there exist arbitrarily small perturbations of the per-capita growth rate functions $f_i$ such that this assumption holds, see Theorem~\ref{thm:generic} which holds for any dimension $n$.

\begin{prop}~\label{prop:2d}
Suppose that $n=2$, and Assumptions~\ref{a.nonde}, \ref{a.extn2} and \ref{a:external} hold. Then exactly one of the following four conclusions holds:
\begin{enumerate}
	\item $\pmu_\by(\bdelta^*)=1$ for all $\by\gg 0,$
	\item there exists an ergodic invariant probability measure $\mu$ such that $|I_\mu|=1$ and $\pmu_\by(\mu)=1$ for all $\by\gg 0$,
	\item there exist ergodic invariant probability measures $\mu_1,\mu_2$ such that $I_{\mu_i}=\{i\}$, $\prod_i \pmu_\by(\mu_i)>0$, and $\sum_i \pmu_\by(\mu_i)=1$ for all $\by\gg0$, or
	\item there exists an ergodic invariant probability measure $\mu$ such that $I_\mu=\{1,2\}$ and $\pmu_\by(\mu)=1$ for all $\by\gg 0$
\end{enumerate}
\end{prop}
\begin{rmk} Propositions~\ref{prop:1d}--\ref{prop:2d} imply that each $\le 2$ dimensional face of $\R^3_+$ supports at most one ergodic invariant probability measure and characterizes the existence of the ergodic measures with the external Lyapunov exponents. \end{rmk}

\begin{rmk}\label{rmk:2d} The four possible outcomes in Proposition~\ref{prop:2d} can be characterized in terms of the external Lyapunov exponents. Case (1) occurs if and only if $\max_i\lambda_i(\bdelta^*)<0$. Case (2) occurs if and only if there exists $i$ and $j\neq i$ such that $\lambda_i(\bdelta^*)>0$, $\lambda_i(\mu_j)<0$ where $I_{\mu_j}=\{j\}$, and either $\lambda_i(\bdelta^*)<0$ or $\lambda_i(\bdelta^*)>0$, $\lambda_j(\mu_i)>0$ where $I_{\mu_i}=\{i\}$. Case (3) occurs if and only if $\lambda_i(\bdelta^*)>0$ for $i=1,2$, and $\lambda_j(\mu_i)<0$ for all $i\neq j$ where $I_{\mu_j}=\{j\}$. Case (4) occurs if and only if there exists $i$ and $j\neq i$ such that $\lambda_i(\bdelta^*)>0$, $\lambda_j(\mu_i)>0$, and either $\lambda_j(\bdelta^*)<0$ or $\lambda_j(\bdelta^*)>0$,$\lambda_i(\mu_j)>0$ where $I_{\mu_j}=\{j\}$.
\end{rmk}

When $n=3$ species, Propositions~\ref{prop:1d}--\ref{prop:2d} characterize the asymptotic behavior of $\BX$ restricted to the one and two-dimensional faces of $\R^3_+.$ To understand the asymptotic behavior of $\BX$ for $\BX(0)\gg 0$, we need to isolate one special form of  $\BX$'s dynamic: the rock-paper-scissors dynamic. This is a type of dynamics where the first species seems to win, grows to significant levels while the other two species have negligible densities. Then species $2$ outcompetes species $1$ and seems to win. After that happens the density of species $2$ decreases and the density of species $3$ increases. Finally, species $1$ wins against species $3$, its density increases and that of species $3$ decreases. Mathematically this scenario corresponds to a stochastic analog of a heteroclinic cycle. An example of an ecosystem with this dynamics is the one including the side-blotched lizard \citep{SL96}. In this ecosystem there are three different types of lizards. The first type is a highly aggressive lizard that attempts to control a large area and mate with any females within the area. The second type is a furtive lizard, which wins against the aggressive lizard by acting like a female. This way the furtive lizard can mate without being detected in an aggressive lizard’s territory. The third type is a guarding lizard that watches one specific female for mating. This prevents the furtive lizard from mating. However, the guarding lizard is not strong enough to overcome the aggressive lizard. This type of dynamics creates regimes where one species seems to win, until the species that beats it makes a comeback. This creates subtle technical problems which we resolve in our proofs.

\begin{deff}\label{deff:rps}  For $n=3$, $\BX$ is a \emph{rock-paper-scissor system} if $\lambda_i(\bdelta^*)>0$ for all $i$, and either
\[
\mbox{\bf(a)} \min\{\lambda_1(\mu_2),\lambda_2(\mu_3),\lambda_3(\mu_1)\}>0>\max\{\lambda_1(\mu_3),\lambda_2(\mu_1),\lambda_3(\mu_2)\}\] or
\[
\mbox{\bf(b)} \max\{\lambda_1(\mu_2),\lambda_2(\mu_3),\lambda_3(\mu_1)\}<0<\min\{\lambda_1(\mu_3),\lambda_2(\mu_1),\lambda_3(\mu_2)\}\]
where $\mu_i$ are the unique, ergodic invariant probability measures satisfying $I_{\mu_i}=\{i\}.$
\end{deff}

\begin{rmk}
Note that if $\lambda_i(\bdelta^*)>0$ then by Proposition \ref{prop:2d} for every $i\in\{1,2,3\}$ there exists a unique ergodic measure $\mu_i$ with $I_{\mu_i}={i}$.
\end{rmk}

The following theorem characterizes, generically, the asymptotic behavior of $\BX$ for $\BX(0)\gg0$ for rock-paper-scissor systems.
\begin{restatable}{thm}{rps}\label{thm:rps}
Assume $n=3$, $\BX$ is a rock-paper-scissor system of type (a), and Assumptions~\ref{a.nonde}--\ref{a.extn2} hold. If
\begin{equation}\label{eq:rps}
\lambda_1(\mu_2)\lambda_2(\mu_3)\lambda_3(\mu_1)+\lambda_1(\mu_3)\lambda_2(\mu_1)\lambda_3(\mu_2)>0,
\end{equation}
then $\BX$ is strongly stochastically persistent. Moreover, if $\mu$ is the ergodic measure such that $\pmu_\by(\mu)=1$ for all $\by\gg0$, then
\begin{equation}\label{e:tv}
\lim\limits_{t\to\infty} \|\PP_\by(\BX(t)\in\cdot)-\pi^*(\cdot)\|_{\text{TV}}=0 \mbox{ for all }\by\gg 0
\end{equation}
where $\|\cdot,\cdot\|_{\text{TV}}$ is the total variation norm.

Alternatively, if the inequality in \eqref{eq:rps} is reversed then
\begin{equation}\label{eq:rps2}
\PP_\by \left(\U \subset \Conv(\{\mu_1,\mu_2,\mu_3\})\mbox{ and }\limsup_{t\to\infty}\frac{1}{t}\log\min_i X_i(t)<0 \right)=1 \mbox{ for all }\by\gg 0
\end{equation}
where $\Conv(\{\mu_1,\mu_2,\mu_3\})$ denotes the convex hull of the probability measures $\{\mu_1,\mu_2,\mu_3\}.$
\end{restatable}

The following theorem characterizes, generically, strong stochastic persistence for $n=3$ for non-rock-paper-scissor systems.
\begin{restatable}{thm}{pers}\label{thm:pers}
Assume $n=3$, $\BX$ is not a rock-paper-scissor system, and Assumptions ~\ref{a.nonde}, \ref{a.extn2}, and \ref{a:external} hold. Then
$\BX$ is strongly stochastically persistent if and  only if $\max_i \lambda_i(\mu)>0$ for all ergodic $\mu$ with $|I_\mu|\le 2.$ Moreover, if $\mu$ is the ergodic measure such that $\pmu_\by(\mu)=1$ for all $\by\gg0$, then \eqref{e:tv} holds.
\end{restatable}

Finally, we characterize what happens $\BX$ is not strongly stochastically persistent and is not a rock-paper-scissor system.
\begin{restatable}{thm}{exclude}\label{thm:exclude}
 Assume $n=3$, $\BX$ is not a rock-paper-scissor system, and Assumptions~\ref{a.nonde}--\ref{a.extn2} and \ref{a:external} hold. If $\BX$ is not stochastically persistent, then there exist ergodic invariant probability measures $\mu^1,\dots,\mu^k$ with $k\le 3 $ such that
\begin{enumerate}
  \item $|I_{\mu^i}|\le 2$ for all $i$,
	\item $I_{\mu^i}\cap I_{\mu^j} \neq I_{\mu^i}$ for all $i\neq j$,
	\item $\prod_{i=1}^k \pmu_\by(\mu^i)>0$ for all $\by\gg0$, and
	\item $\sum_{i=1}^k \pmu_\by(\mu^i)=1$ for all $\by\gg 0.$
\end{enumerate}
\end{restatable}
\begin{rmk}
We can actually prove the stronger result which says that extinction is exponentially fast with rate given by the relevant external Lyapunov exponent $$\pmu_\by(\mu^\ell):=\PP_\by\left\{\U=\{\mu^\ell\}\,\text{ and }\,\lim_{t\to\infty}\dfrac{\ln X_i(t)}t=\lambda_i(\mu^\ell)<0, i\notin I_{\mu^\ell} \right\}>0, \by\gg0, \ell=1,\dots,k.$$
\end{rmk}
\begin{figure}
\includegraphics[width=0.3\textwidth]{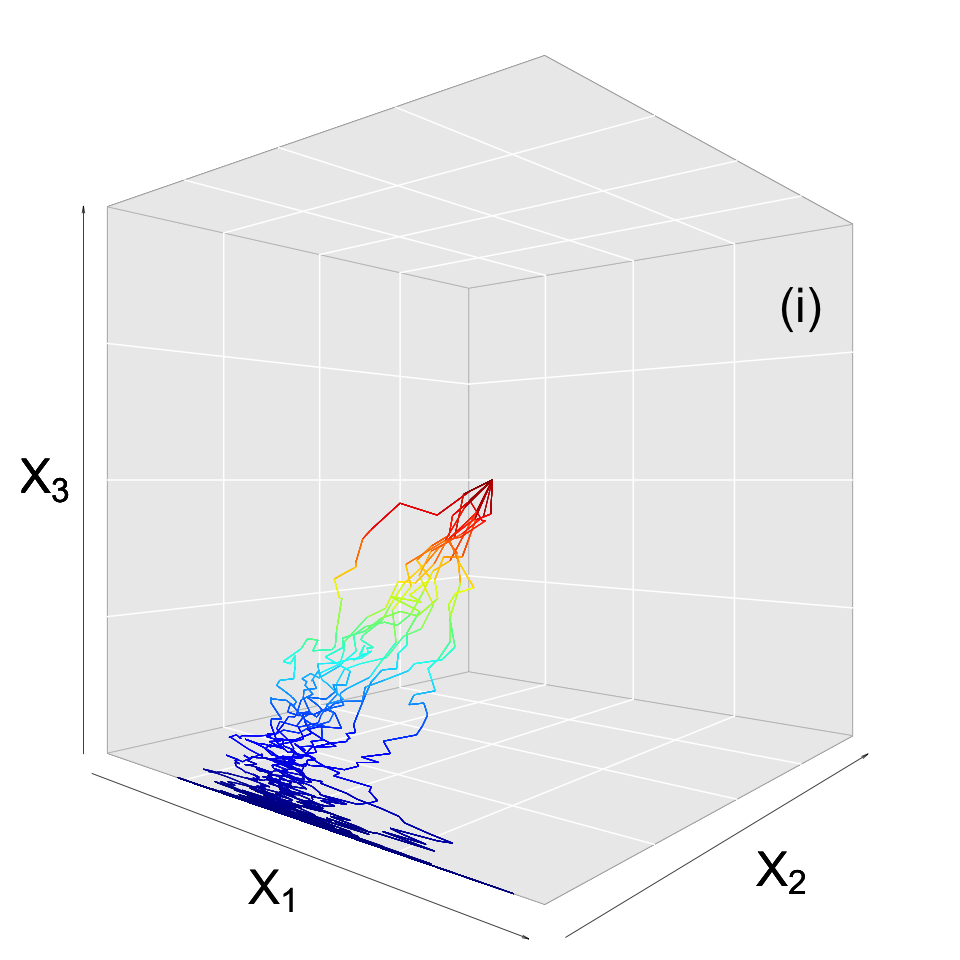}
\includegraphics[width=0.3\textwidth]{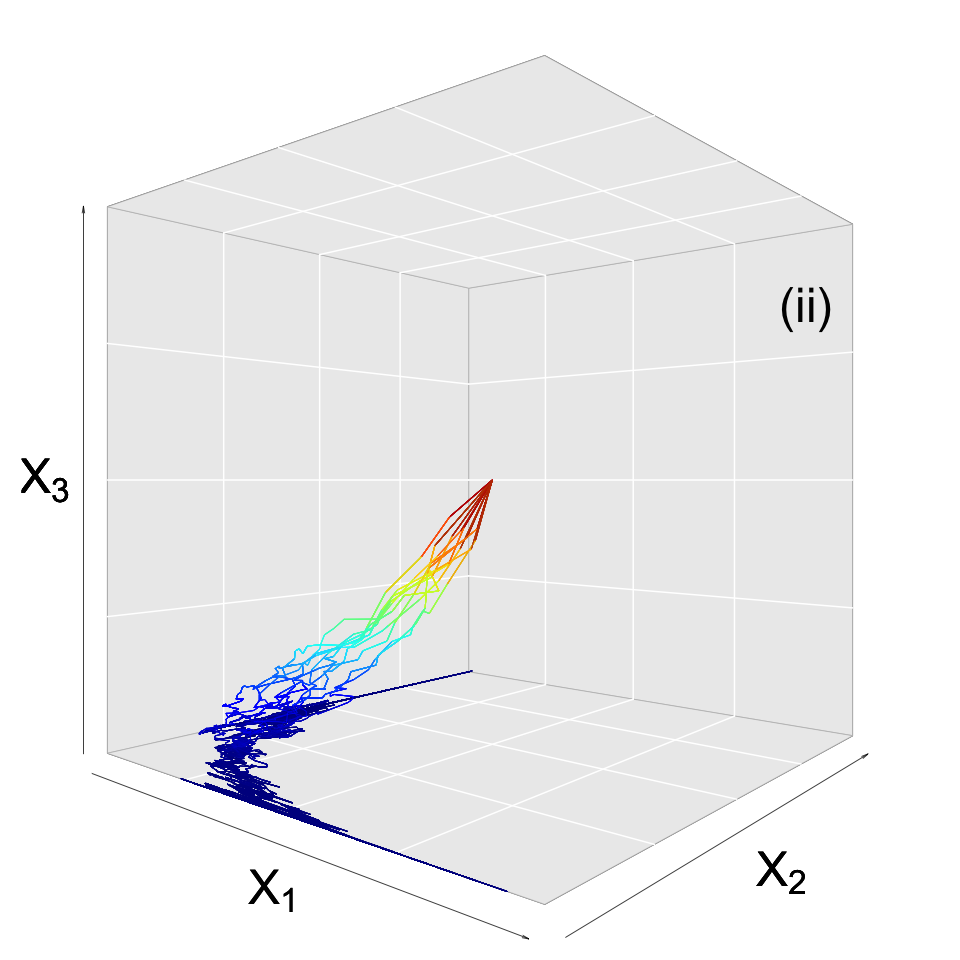}
\includegraphics[width=0.3\textwidth]{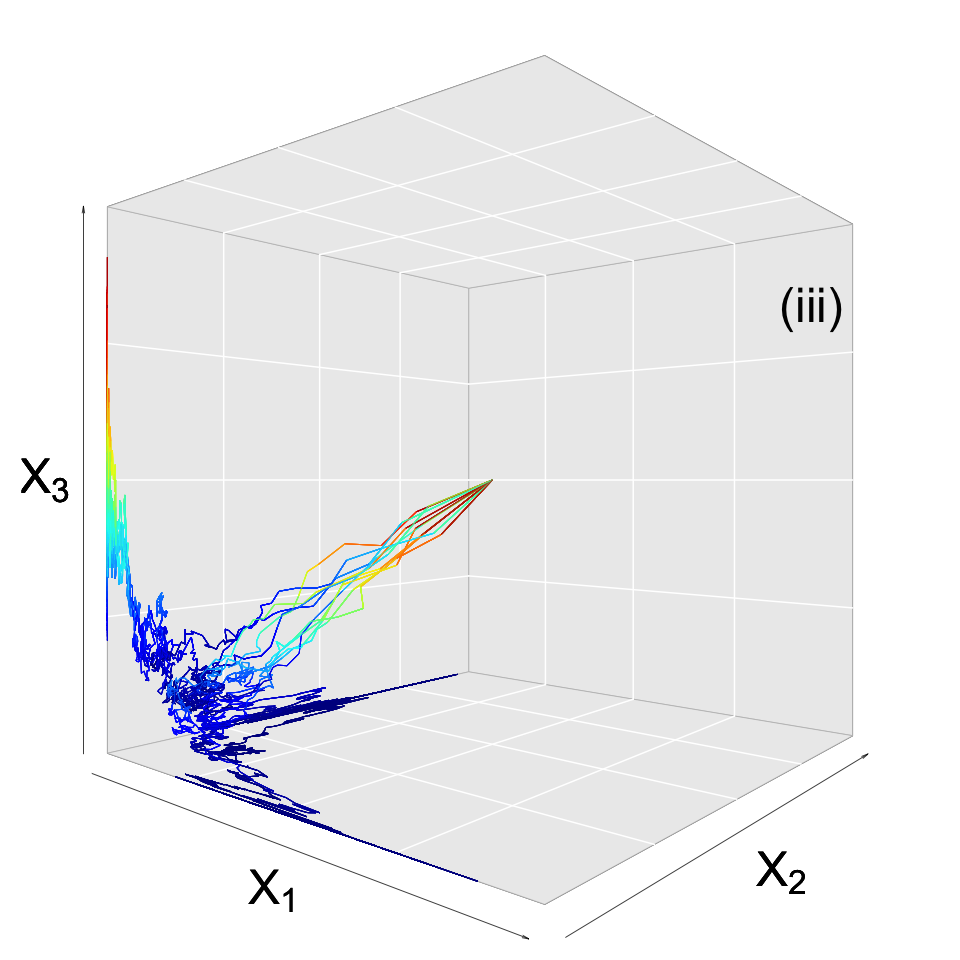}\\
\includegraphics[width=0.3\textwidth]{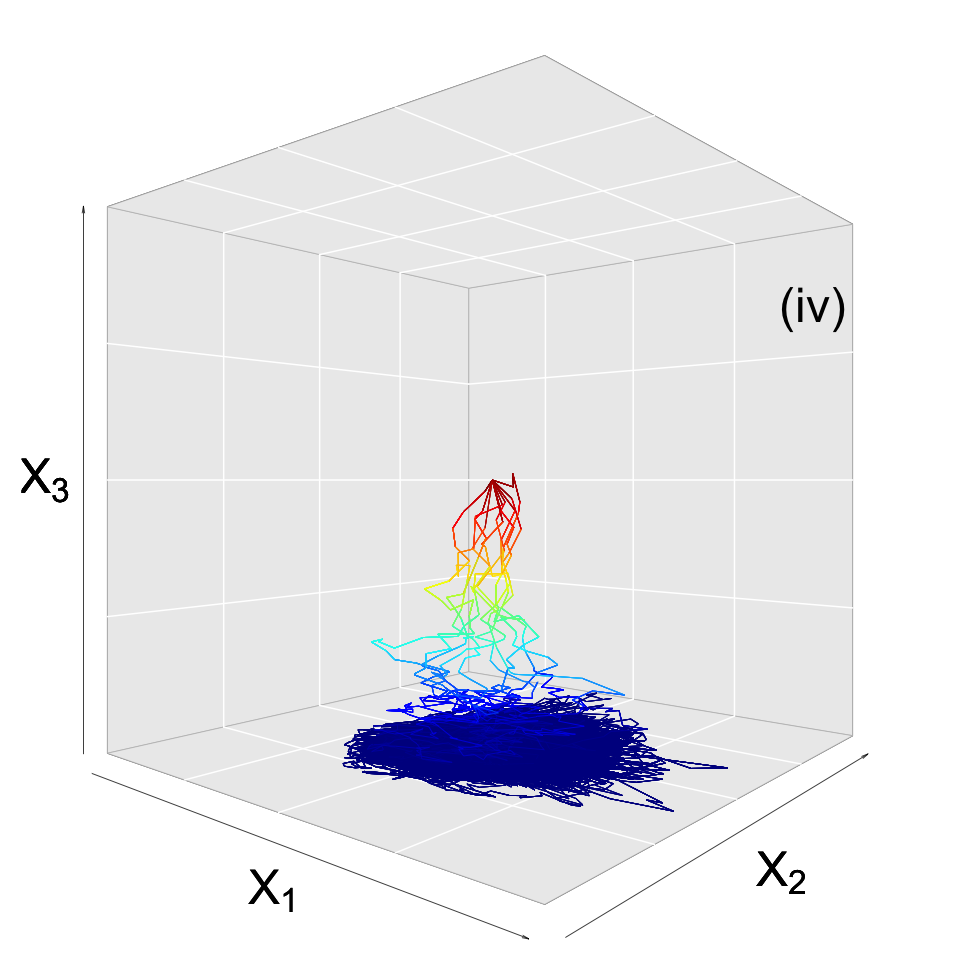}
\includegraphics[width=0.3\textwidth]{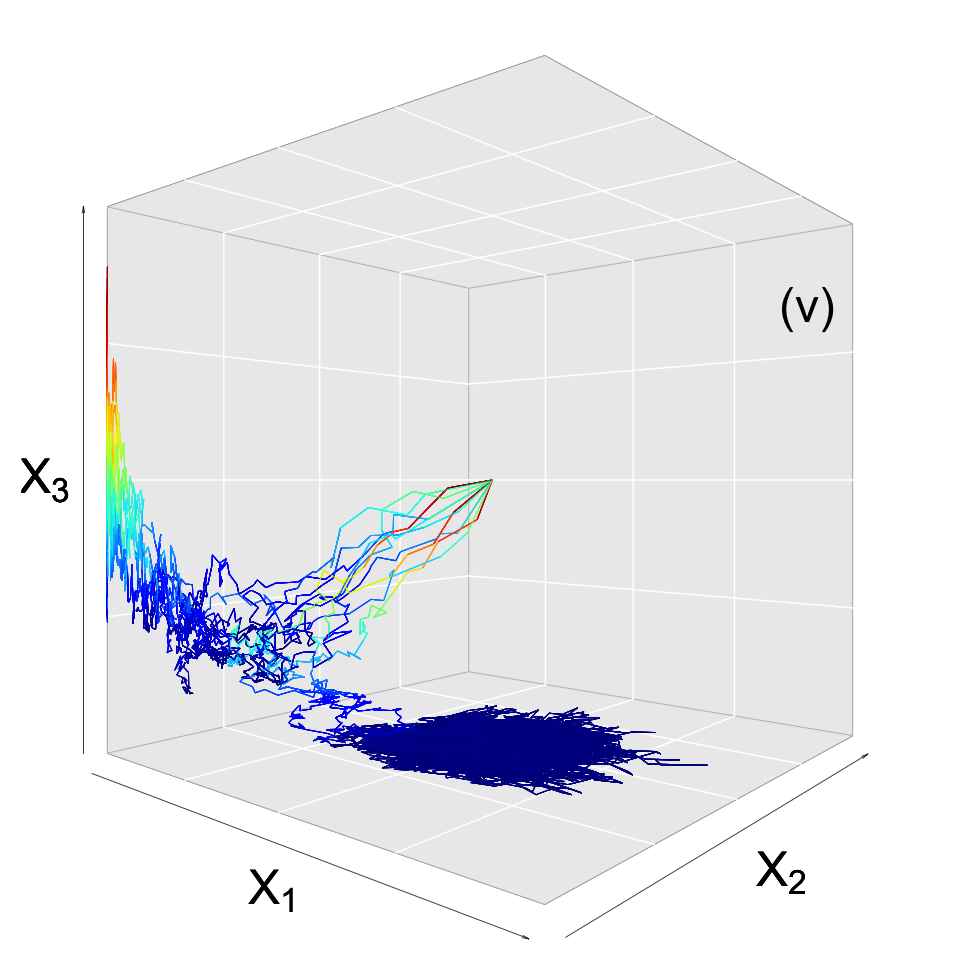}
\includegraphics[width=0.3\textwidth]{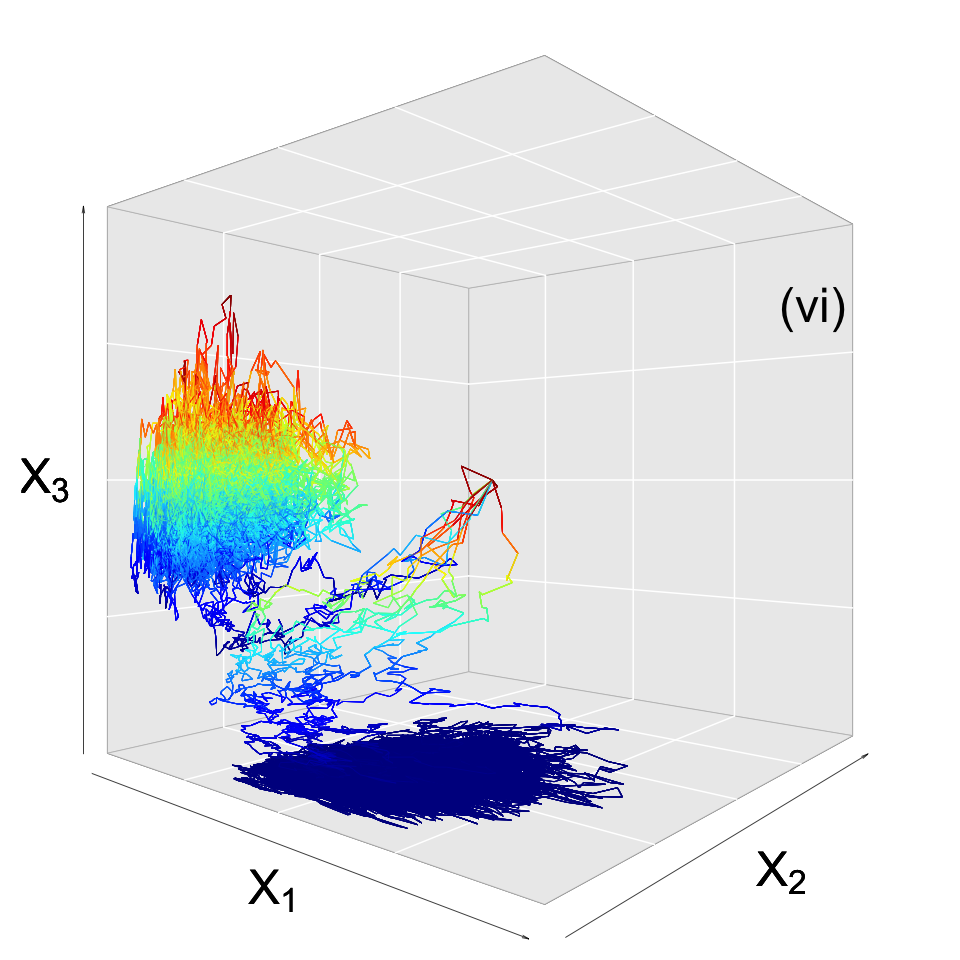}\\
\includegraphics[width=0.3\textwidth]{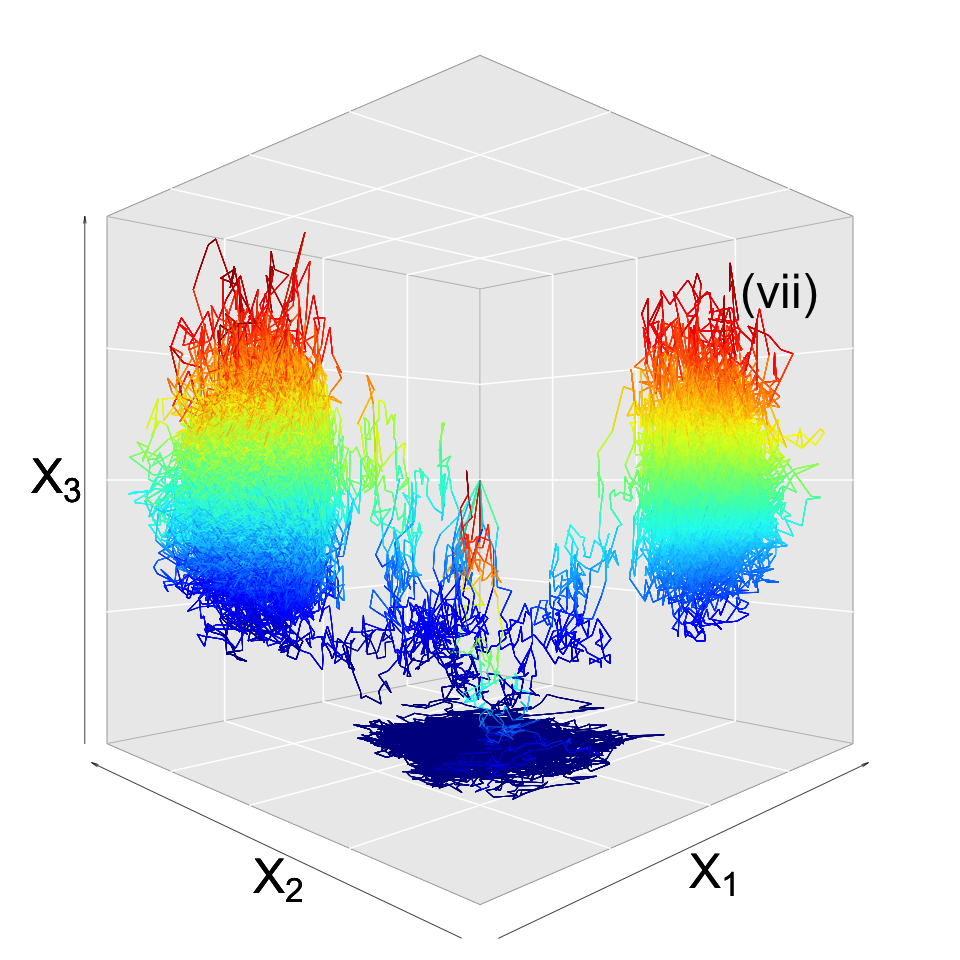}
\includegraphics[width=0.3\textwidth]{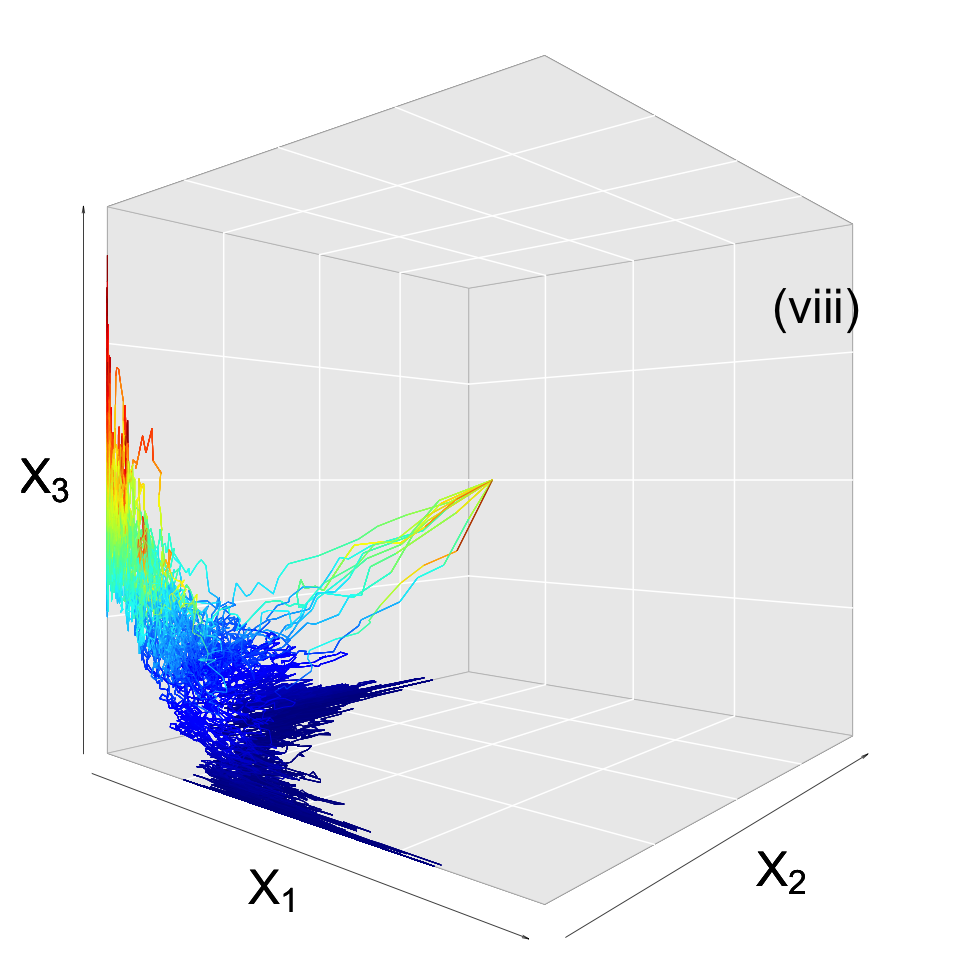}
\includegraphics[width=0.3\textwidth]{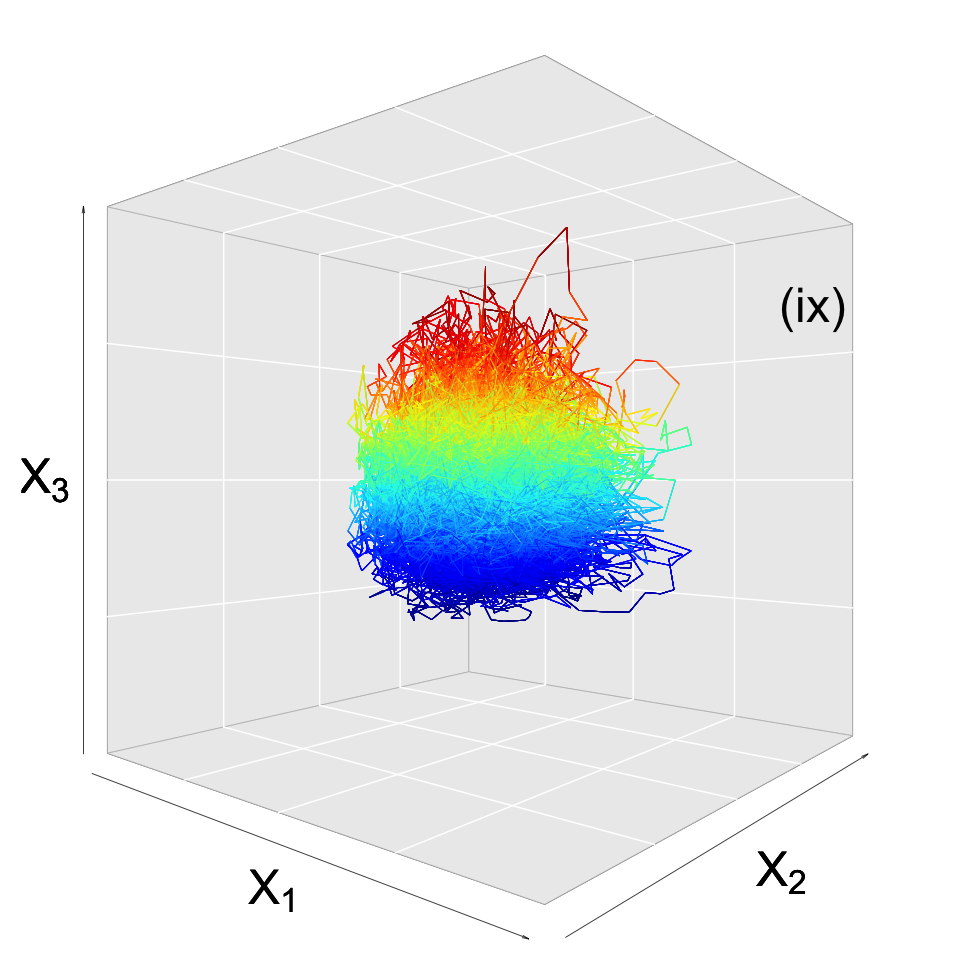}
\caption{The nine types of stochastic dynamics, up to permutation of indices, where at least one species persists. Drift functions are (i) $f(\bx)=(1-x_1,1-2x_1-x_2,1-2x_1-x_3)$, (ii) $f(\bx)=(1-x_1-2x_2,1-2x_1-x_2,1-2x_1-2x_2-x_3)$, (iii) $f(\bx)=(1-x_1-2x_2-2x_3,1-2x_1-x_2-2x_3,1-2x_1-2x_2-x_3)$, (iv) $f(\bx)=(1-x_1,1-x_2,1-x_1-x_2-x_3)$, (v) $f(\bx)=(1-x_1-2x_3,1-x_2-2x_3,1-x_1-x_2-x_3)$, (vi) $f(\bx)=(1-x_1-2x_3,-0.1+0.4x_1-0.5x_2+0.4x_3,1-2x_1-x_2)$, (vii) $f(\bx)=(1-x_1-4x_2x_3,1-x_2-4x_1x_3,1-x_3-x_1x_2)$, (viii) $f(\bx)=c(1-x_1-2x_2-0.8x_3,1-0.8x_1-x_2-2x_3,1-2x_1-0.8x_2-x_3)$, (ix)$f(\bx)=c(1-x_1,1-x_2,1-x_3)$. The diffusion term for species $i=1,2,3$ is $0.25X_idB_i(t)$ where $B_1(t),B_2(t),B_3(t)$ are independent, standard Brownian motions.}\label{fig:one}
\end{figure}

Up to permutations of the indices, these theorems characterize the asymptotic behavior of $\BX$ for $\BX(0)\gg 0$ into $10$ types. One type corresponds to all species going extinct, the other $9$ types where at least one species persists are shown in Figure~\ref{fig:one}. As shown in the proofs of the Theorems, all $10$ types of dynamics are characterized by the external Lyapunov exponents. For example, the case of $\mu^1,\mu^2,\mu^3$ with $I_{\mu^i}=\{i\}$ for Theorem~\ref{thm:exclude} occurs if and only if $\lambda_i(\bdelta^*)>0$ for all $i$, and $\max_{j\neq i}\lambda_j(\mu^i)<0$ for all $i$. Alternatively, the case of $\mu^1,\mu^2$ with $I_{\mu^1}=\{1,2\}$ and $I_{\mu^2}=\{3\}$ for Theorem~\ref{thm:exclude} occurs if and only if $\BX$ restricted to the first two species satisfies the strongly persistent condition (see, Remark~\ref{rmk:2d}), $\lambda_3(\mu^1)<0$, $\lambda_3(\bdelta^*)>0$, and $\max_{i=1,2}\lambda_i(\mu^2)<0$.

Finally, we show that Assumption~\ref{a:external} (i.e. all external Lyapunov exponents are non-zero) holds generically. In order to measure how far apart processes are from each other we need to define a topology on the stochastic differential equations~\ref{e:system}. To this end, we make the following definition.
\begin{deff} A process $\tilde \BX$ satisfying
\begin{equation}\label{e:system_2}
d\tilde X_i(t)=\tilde X_i(t) \tilde f_i(\tilde \BX(t))dt+\tilde X_i(t)\tilde g_i(\tilde\BX(t))dE_i(t), ~i=1,\dots,n
\end{equation}
and Assumptions \eqref{a.nonde} and \eqref{a.extn2} is a \emph{$\delta$-perturbation of \eqref{e:system}} for some $\delta>0$ if
\begin{equation}\label{robust1}
\sup_{\bx\in \R_+^n}\| \tilde f_i(\bx)- f_i(\bx)
\| + \sup_{\bx\in \R_+^n} \|\tilde g_i(\bx)- g_i(\bx) \| \leq \delta.
\end{equation}
\end{deff}

\begin{thm}\label{thm:generic}
Suppose \eqref{e:system} satisfies Assumptions \ref{a.nonde} and \ref{a.extn2}. For any $\delta>0$, there exist functions $\tilde f, \tilde g=g$ defining a process $\tilde \BX(t)$ by \eqref{e:system_2} such that
\begin{enumerate}
  \item $\tilde \BX(t)$ is a $\delta$-perturbation of $\BX(t)$,
  \item For every ergodic measure of $\tilde \BX(t)$ the external Lyapunov exponents are non-zero.
\end{enumerate}
\end{thm}

We note that the set of ergodic measures of the perturbed process $\tilde \BX(t)$ in Theorem~\ref{thm:generic} need not equal the ergodic measures of the unperturbed process $\BX(t)$.

\section{Proofs of Theorem \ref{thm:generic} and Propositions \ref{p:tight}, \ref{prop:LV}, \ref{p:rps}, \ref{prop:switch:ode}, \ref{prop:switch:sde}}\label{s:app1}
\begin{proof}[Proof of Theorem \ref{thm:generic}]
Let $\delta>0$ be given. To achieve the desired perturbation, we create a sequence of perturbations $\tilde f^0,\dots, \tilde f^{n-1}$ such that for all $0\le k \le n-1$, (i) $\tilde f^k, \tilde g=g$ is a $\delta$-perturbation  of $f,g$, (ii) for every ergodic invariant probability measure with $|I_\mu|\le k$, $\lambda_i(\mu)\neq 0$ for all $i\notin I_\mu$, and (iii) for $k\ge 1$, $\tilde f^k(\bx)=\tilde f^{k-1}(\bx)$ for all $\bx$ with $x_k=x_{k+1}=\dots=x_n=0$. Note that condition (iii) ensures that the processes associated with the $(f^{k-1},g)$ and $(f^k,g)$ perturbations have the same set of ergodic probability measures supported by the set $\{\bx:x_k=x_{k+1}=\dots=x_n=0\}$. Outside of this set, the ergodic probability measures of these two processes may not be the same. We prove the existence of this sequence inductively.

For $k=0$, the only ergodic invariant probability measure $\mu$ with $|I_\mu|=0$ is $\bdelta^*$. For the species $i$ such that $\lambda_i(\bdelta^*)\neq 0$, define $\tilde f_i^0=f_i$. For any species $i$ for which $\lambda_i(\bdelta^*)=0$, define\[
\tilde f_i^0(\bx)=f_i(\bx)-\frac{\delta}{2} \phi^0_i(\bx)
\]
where $\phi^0_i$ is a smooth, non-negative function that is $1$ at the origin, $0$ outside a small neighborhood of the origin, and $\|\phi^0_i\|_\infty =1$. After the perturbation
\[
\lambda_i(\bdelta^*)=-\int \frac{\delta}{2} \phi^0_i(x) \bdelta^*(d\bx) =-\frac{\delta}{2} <0,
\]
$\tilde f^0$ satisfies (i)--(iii).

Now assume there exist $\tilde f^0,\dots, \tilde f^k$ that satisfy (i)--(iii) and $k\le n-2$. We will construct $\tilde f^{k+1}$ that satisfies (i)--(iii). By Assumption~\ref{a.nonde}, for each $I\subset\{1,\dots,n\}$ there exists at most one ergodic invariant probability measure $\mu$ such that $I_\mu=I$. Let $J\subset \{1,\dots,n\}$ be the collection of $i$s such that $\lambda_i(\mu)\neq 0$ for any ergodic invariant probability measure $\mu$ with $|I_\mu|=k+1$ and $i\notin I_\mu$. For $i\in J$, define $\tilde f^{k+1}_i=\tilde f^k_i.$ For $i\notin J$, let $\mathcal{M}_i$ be the (finite) set of ergodic invariant probability measures $\mu$ such that $i\notin I_\mu, |I_\mu|=k+1$, and $\lambda_i(\mu)=0$. Let $e_1,\dots,e_n$ be the canonical basis vectors and set $\mathcal{M}_i:=\{\mu_i^1,\dots,\mu_i^\ell\}$ to be an order of $\mathcal{M}_i$. Do the following procedure in order from $\mu_i^1$ up to $\mu_i^\ell$. For $\mu_i^j\in \mathcal{M}_i$, let $\phi^{\mu_i^j}(\bx)$ be a smooth function taking values in $[0,1]$ such that $\phi^{\mu_i^j}(\sum_{i\in I_\mu}e_i)=1$, and the support of $\phi^{\mu_i^j}$ doesn't intersect any of the $\le k$ dimensional faces of $\partial \R^n_+$ nor the support of any of the previously defined $\phi$ functions. Define $\tilde f^{k+1}(\bx)=\tilde f^k(\bx)-\frac{\delta}{2}\sum_{i=1}^\ell \phi^{\mu_i^j}(\bx).$ Then $\lambda_i(\mu)=-\frac{\delta}{2}\int \phi_i^\mu(\bx)\mu(\bx)<0$ for all $\mu \in \mathcal{M}_i$. Note that, since the $\phi$'s have compact support, the perturbations of the drift terms will not violate Assumptions \ref{a.nonde} or \ref{a.extn2}. By construction,  $\tilde f^{k+1}$ satisfies (i)--(iii).

Let $\tilde \BX(t)$ be the solution of
\begin{equation*}
d\tilde X_i(t)=\tilde X_i(t) \tilde f^{n-1}(\tilde \BX(t))dt+\tilde X_i(t) g_i(\tilde\BX(t))dE_i(t), ~i=1,\dots,n.
\end{equation*}
Then $\tilde \BX(t)$ is a $\delta$-perturbation of $\BX(t)$ that has no zero external Lyapunov exponents.
\end{proof}

\begin{proof}[Proof of Proposition \ref{p:tight}]
If the system is competitive, so that $a_{ij}\leq 0$ for all $i,j=1,2,3$ then Example 1.1 from \cite{HN16} proves that such a triplet $(c_1,c_2,c_3)\in\R_+^{3,\circ}$ exists.

Suppose that $a_{12}<0, a_{13}\leq 0$ and $a_{23}\leq0$. In particular, this treats, after possibly reordering the indices, all the combinations of predator-prey and competitive interactions.
Let $c_1=M\dfrac{|a_{21}|+1}{|a_{12}|}, c_2=M>0, c_3=1$ and note that
$$
\begin{aligned}
\sum c_ix_if_i(x)\leq&
\left(M\dfrac{|a_{21}|+1}{|a_{12}|}m_1x_1+Mm_2x_2+m_3x_3\right)\\
&+M\dfrac{|a_{21}|+1}{|a_{12}|}(a_{11}x_1^2-|a_{12}|x_1x_2-|a_{13}|x_1x_3)\\
&+M(a_{22}x_2^2+a_{21}x_1x_2-|a_{23}|x_2x_3)+(a_{33}x_3+a_{32}x_1x_3+a_{32}x_2x_3)\\
\leq& \left(M\dfrac{|a_{21}|+1}{|a_{12}|}m_1x_1+Mm_2x_2+m_3x_3\right)\\
&+M\dfrac{|a_{21}|+1}{|a_{12}|}a_{11}x_1^2+Ma_{22}x_2^2
+a_{33}x_3^2+a_{32}x_1x_3+a_{32}x_2x_3\\
\leq& \left(M\dfrac{|a_{21}|+1}{|a_{12}|}m_1x_1+Mm_2x_2+m_3x_3\right) +\dfrac{a_{33}}2(x_1^2+x_2^2+x_3^2)\, \text{ (for sufficiently large $M$)}\\
\leq& K_M +\dfrac{a_{33}}2(x_1^2+x_2^2+x_3^2)\, \text{ (for sufficiently large $K_M$)}.\\
\end{aligned}
$$
Note that here we use the fact that $a_{ii}<0, i=1,2,3$.
As a result
$$
\dfrac{\sum c_ix_if_i(x)}{1+\bc^\top\bx}\leq \gamma_1 -\gamma_2(x_1+x_2+x_3)
$$
for some constants $\gamma_1,\gamma_2>0$.
Since $g_i(\bx)=1$, it is easy to see that
\eqref{a.tight} holds.
\end{proof}

\begin{proof}[Proof of Proposition \ref{prop:LV}] By Proposition \ref{p:zero} we have that for all $i\in I_\mu$
\[
0=\lambda_i(\mu)=\int_{\R^3_+}\left(m_i+\sum_{j=1}^3 a_{ij} x_j-\dfrac{\sigma_{ii}}2\right)\mu_i(d\bx) = m_i+
\sum_j a_{ij} \int_{\R^3_+} x_j\mu(d\bx)-\dfrac{\sigma_{ii}}2.
\]
By assumption, there exists a unique solution $\bar \bx$ to \eqref{e:sys}. Hence, $\bar x_i = \int_{\R^3_+} x_i \mu(d\bx)$ for all $i$ and the claimed expression for $\lambda_i(\mu)$ follows.
\end{proof}

\begin{proof}[Proof of Proposition \ref{p:rps}]
If $2>\alpha+\beta$, we have
\begin{equation}\label{e:SDE22}
\begin{split}
\lambda_2(\mu_1)\lambda_3(\mu_2)\lambda_1(\mu_3)&=\left(\frac{\sigma^3}{2 a_{11}a_{22}a_{33}}\right)(a_{11}-a_{21} )(  a_{22}-a_{32})(a_{33}-a_{13}) \\
&= \left(\frac{\sigma^3}{2 a_{11}a_{22}a_{33}}\right) (1-\beta)^3\\
&>\left(\frac{\sigma^3}{2 a_{11}a_{22}a_{33}}\right) (\alpha-1)^3\\
&= \left(\frac{\sigma^3}{2 a_{11}a_{22}a_{33}}\right) | (a_{22}-a_{12})( a_{33}-a_{23})( a_{11}-a_{31})| \\
&= |\lambda_3(\mu_1)\lambda_1(\mu_2)\lambda_2(\mu_3)|
\end{split}
\end{equation}
and by Theorem \ref{thm:rps} there is persistence. If $2<\alpha+\beta$ then from Theorem \ref{thm:rps} we have that with probability one $$\BX(t)\to \partial \R_+^3$$
as $t\to\infty$ and there is extinction.
\end{proof}

\begin{proof}[Proof of Proposition \ref{prop:switch:ode}]
The proof follows  from the proofs of Case C in Theorems 3.1 and 3.2 of \citet{hutson1984} (see also \citet{hutson_law1985}). The only difference between \eqref{e:switch:ode} and the models considered by \citet{hutson1984} is that our model includes the self-limitation term $-cX_3$ in the predator equation. The proofs of \citet{hutson1984} imply that permanence occurs if all the equilibria on the boundary $\partial \R^3_+$ have at least one positive external Lyapunov exponent with respect to the Dirac measure at the equilibrium. Alternatively, if all the external Lyapunov exponents are negative at one of the boundary equilibria, say $\bx^*$, then  there are positive initial conditions $\BX(0)\gg0$ such that $\lim_{t\to\infty} \BX(t)=\bx^*$ i.e. the system is impermanent.

The external Lyapunov exponents of the prey species at the origin are given by $\lambda_1(\bdelta^*)=\lambda_2(\bdelta^*)=r>0$. The only additional equilibria on the axes are given by $(r,0,0)$ and $(0,r,0)$ at which the predator's per-capita growth rate (the external Lyapunov exponent) equals $\lambda_3=r-d$ and the missing prey's per-capita growth rate equals $r(1-\beta)<0$. Hence, there is a positive external Lyapunov exponent at these equilibria if and only $r>d$ . The only other equilibrium in the $X_1$--$X_2$ plane is the unstable equilibrium $(r/(1+\beta),r/(1+\beta),0)$. At this equilibrium, the external Lyapunov exponent of the predator equals $r/(1+\beta)-d$ which is positive if and only if $r>(1+\beta)d$. When $r>d$, there are the equilibria $((rc+d)/(1+c),0,(r-d)/(1+c))$ and $(0,(rc+d)/(1+c),(r-d)/(1+c))$ in the $X_1$--$X_3$ and $X_2$--$X_3$ planes, respectively. The external Lyapunov exponent at these equilibria equal $r-\beta (rc+d)/(1+c)$. Hence, permanence occurs if and only if $r-\beta(rc+d)/(1+c)>0$ and $r/(1+\beta)>d$ which are equivalent to the stated conditions for permanence.
\end{proof}

\begin{proof}[Proof of Proposition \ref{prop:switch:sde}]
We begin by noting that while the functions $f_i(\bx)$ in \eqref{e:switch:sde} are not locally Lipschitz when $x_1=x_2=0$, the full drift functions $g_i(\bx)=x_if_i(\bx)$ can be uniquely extended to be locally Lipschitz functions at $x_1=x_2=0$ by defining $g_1(\bx)=g_2(\bx)=0$ and $g_3(\bx)=-d-cx_3.$ Hence, there is existence and uniqueness of strong solutions. Moreover, Theorem~\ref{thm:pers} still holds by making the change of coordinates $S=x_1+x_2$ and $y=x_1/S$ which by It\^{o}'s lemma yields
\[
\begin{aligned}
dS(t)=& S(t)\left(r-(1+\beta)S(t)-(y(t)^2+(1-y(t))^2)X_3(t)\right)dt+\varepsilon S(t)\left(y(t)dB_1(t)+(1-y(t))dB_2(t) \right)\\
dy(t)=& y(t)(1-y(t))(1-2y(t))\left(S(t)(1-\beta)+X_3(t)+\varepsilon^2 \right)dt +y(t)(1-y(t))\varepsilon(dB_1(t)-dB_2(t))\\
dX_3(t)=& X_3(t)\left(S(t)(y(t)^2+(1-y(t))^2-d-cX_3(t)  \right)dt+\varepsilon X_3(t)dB_3(t),
\end{aligned}
\]
and applying the arguments in Section~\ref{s:predprey} to this system whose state space is $[0,\infty)\times [0,1] \times [0,\infty)$ and where extinction of one or more species corresponds to $y(1-y)Sx_3=0$.

As Theorem~\ref{thm:pers} applies, we will identify when every ergodic invariant probability measure on the boundary has at least one positive external Lyapunov exponent. For the Dirac measure at the origin, $\lambda_i(\bdelta^*)=r-\varepsilon^2/2$ for $i=1,2$. Assume $0<\varepsilon<\sqrt{2r}$. Proposition~\ref{prop:1d} implies that for $i\in \{1,2\}$ there is a unique ergodic measures $\mu_i$ such that $I_{\mu_i}=\{i\}$. As the Lyapunov exponent $\lambda_3(\bdelta^*)=-d-\varepsilon^2$ is negative, Proposition~\ref{prop:1d} implies there is no additional ergodic invariant probability measure on the $x_3$ axis. The unique solution $\bar x_i$ for $i=1,2$ to $0=r-\bar x_i-\varepsilon^2/2$ is $\bar x_i=r-\varepsilon^2/2$. Using Proposition~\ref{prop:LV} we therefore get $\bar x_i = \int x_i \mu_i(d\bx)$ for $i=1,2.$ The external Lyapunov exponents at $\mu_i$ are $\lambda_j(\mu_i)=r-\beta \bar x_i -\varepsilon^2/2<0, j\in\{1,2\}\setminus\{i\}$ for the other prey species and $\lambda_3(\mu_i)=\bar x_i -d-\varepsilon^2/2=r-d-\varepsilon^2$. In the $x_1$$x_2$ plane, the negative external Lyapunov exponents for $\mu_1,\mu_2$ and Proposition~\ref{prop:2d} imply that there are no ergodic invariant probability measures $\mu$ with $I_\mu=\{1,2\}.$

Assume that the external Lyapunov exponents $\lambda_3(\mu_i)=r-d-\varepsilon^2/2$ are positive. Proposition~\ref{prop:2d} implies there exists a unique ergodic invariant probability measure $\mu_{i3}$ such that $I_{\mu_{i3}}=\{i,3\}$ for $i=1,2.$ Solving the linear equations $r-\hat x_1 - \hat x_3-\varepsilon^2/2=0=\hat x_1-d-c\hat x_3 - \varepsilon^2/2$ for $\hat x_1,\hat x_3$ yields $\hat x_3= (r-d-\varepsilon^2)/(1+c)$ and $\hat x_1= (rc+d+\varepsilon^2)/(1+c) -\varepsilon^2/2$. Proposition~\ref{prop:LV} implies that
\[
\lambda_1(\mu_{13})=\lambda_2(\mu_{23})=r-\beta\left( (rc+d+\varepsilon^2)/(1+c) -\varepsilon^2/2\right).
\]
For $\varepsilon>0$ sufficiently small, $\lambda_1(\mu_{13})>0$ if $\frac{r}{\beta}(1+c(1-\beta))>d$ in which case Theorem~\ref{thm:pers} implies the system is strongly stochastically persistent.
\end{proof}

\section{Proofs of Theorems \ref{thm:pers} and \ref{thm:exclude}}\label{s:app2}

To prove Theorems \ref{thm:pers} and \ref{thm:exclude}, we make use of two key results from \cite{HN16}. The first result provides a sufficient condition for strong, stochastic persistence in terms of the external Lyapunov exponents. The second result provides a sufficient condition for $\pmu_\by(\mu)>0$ for $\by\gg 0$ and an ergodic measure $\mu\in \M$ supporting a subset of species. These results, however, do not cover two special cases. The first of these special cases corresponds to two prey-single predator systems. For this special case, the sufficient condition of \cite{HN16} for stochastic persistence does not apply. Hence, Theorem \ref{thm2.1} in Section \ref{s:predprey} provides the necessary and sufficient condition (under the assumption of non-zero external Lyapunov exponents) for stochastic persistence. The second special case corresponds to rock-paper-scissor systems as defined in Definition \ref{deff:rps}. For this special case, the condition for the boundary to be attracting doesn't follow from \cite{HN16}. Hence, Theorem \ref{thm4.1} from Section \ref{s:rock-paper-scissors} provides the necessary result for this case.

Let $\M$ be the set of ergodic invariant probability measures of $\BX$ supported on the boundary $\partial\R^3_+:=\R_+^3\setminus \R_+^{3,\circ}$. Denote by $\Conv(\M)$ the invariant probability measures supported on $\partial\R^3_+$,
i.e. the probability measures $\pi$ of the form
$\pi(\cdot)=\sum_{\nu\in\M}p^\nu\nu(\cdot)$
with $p^\nu\geq 0,\sum_{\nu\in\M}p^\nu=1$.

The following condition ensures strong stochastic persistence.
\begin{asp}\label{a.coexn}
For any $\mu\in\Conv(\M)$ one has
$$\max_i\lambda_i(\mu)>0.$$
\end{asp}
We note \citep{SBA11, HN16, BS18} that Assumption \ref{a.coexn} is equivalent to the following assumption.
\begin{asp}\label{a:hof}
There exist numbers $p_i\geq0$ such that
\[
\sum_{i}p_i\lambda_i(\mu)>0, \mu\in \M.
\]
\end{asp}
\begin{thm}\label{t:pers2}
Suppose that Assumptions \ref{a.nonde} and \ref{a.coexn} hold. Then $\BX$ is strongly stochastically persistent and converges exponentially fast to a unique invariant probability measure $\pi^*$ which is supported on $\R_+^{3,\circ}$.
\end{thm}

\begin{proof}
This follows by Theorem 1.1 from \cite{HN16}.
\end{proof}
\zero*
\begin{proof}
 This follows by \cite{HN16}.
\end{proof}

\begin{asp}\label{a.extn4}
There exists an ergodic measure $\mu\in\M$ such that
\begin{equation}\label{ae3.1}
\max_{i\in I_\mu^c}\lambda_i(\mu)<0.
\end{equation}
If $\R_+^\mu\ne\{\0\}$, suppose further that
for any $\nu\in\Conv(\M_\mu)$, we have
\begin{equation}\label{ae3.2}
\max_{i\in I_\mu}\lambda_i(\nu)>0
\end{equation}
where $\M_\mu:=\{\nu'\in\M:\suppo(\nu')\subset\partial\R_+^\mu\}.$
\end{asp}
We call an ergodic measure satisfying Assumption \ref{a.extn4} a \textit{transversal attractor}. This means that $\mu$ attracts all directions that are not among the directions from its support $I_\mu$. Note that by Proposition \ref{p:zero} we always have $\lambda_i(\mu)=0, i\in I_\mu$. Assumption \ref{a.extn4} says that there exists at least one transversal attractor. Define
\begin{equation}\label{e:M1}
\M^1:=\left\{\mu\in\M : \mu ~~\text{satisfies Assumption} ~~\ref{a.extn4}\right\},
\end{equation}
and
\begin{equation}\label{e:M2}
\M^2:=\M\setminus\M^1.
\end{equation}
We need an additional assumption which ensures that apart from those in $\Conv(\M^1)$,
invariant probability measures are \textit{repellers}.
\begin{asp}\label{a.extn3}
Suppose that one of the following is true
\begin{itemize}
  \item $\M^2=\emptyset$

  \item For any $\nu\in\Conv(\M^2)$, $\max_i \lambda_i(\nu)>0.$
\end{itemize}
\end{asp}

\begin{thm}\label{t:ex2}
Suppose that Assumptions \ref{a.nonde},  \ref{a.extn2}, \ref{a.extn4} and \ref{a.extn3} are satisfied and $\M^1\neq \emptyset$.
Then for any $\bx\in\R^{3,\circ}_+$
\begin{equation}\label{e0-thm4.2}
\sum_{\mu\in\M^1} \pmu_\bx(\mu)=1
\end{equation}
where
\[
\pmu_\bx(\mu)=\PP_\bx\left(\U=\{\mu\} \mbox{ and } \limsup_{t\to\infty}\frac{1}{t}\log X_i(t)=\lambda_i(\mu) \mbox{ for all }i\notin I_\mu \right).
\]

\end{thm}
\begin{proof}
This follows from Theorem 1.3 in \cite{HN16}.
\end{proof}

\pers* 
\begin{proof}
Suppose we are in the setting from Section \ref{s:predprey}. This means that there are two prey species and one predator such that:

$$\lambda_1(\bdelta^*)>0, \lambda_2(\bdelta^*)>0, \lambda_3(\bdelta^*)<0,$$
$$\lambda_2(\mu_1)<0, \lambda_1(\mu_2)<0, \lambda_3(\mu_1)>0, \lambda_3(\mu_2)>0,$$
and
$$\lambda_2(\mu_{13})>0, \lambda_1(\mu_{23})>0.$$
In this special case the result follows from Theorem \ref{thm2.1} below.

Suppose that we are not in the setting from Section \ref{s:predprey} or in the rock-paper-scissors setting from Section \ref{s:rock-paper-scissors}. Then one can check, case by case like we do in Section \ref{s:results}, that
\[
\max_i \lambda_i(\mu)>0, \mu\in\M
\]
is equivalent to the existence of $p_i\geq 0$ such that
\[
\sum_{i}p_i\lambda_i(\mu)>0, \mu\in \M
\]
which is equivalent to Assumption \ref{a.coexn}. This allows us to use Theorem \ref{t:pers2} and finish the proof.
\end{proof}
\exclude*
\begin{proof}
This follows from Theorem \ref{t:ex2} by noting that Assumptions \ref{a.extn4} and \ref{a.extn3} hold.
\end{proof}

\subsection{Two prey and one predator}\label{s:predprey}

Throughout this subsection we make the following assumption.
\begin{asp}\label{a:2p1p}
 There are two prey species $1,2$ and one predator $3$ such that:
$$\lambda_1(\bdelta^*)>0, \lambda_2(\bdelta^*)>0, \lambda_3(\bdelta^*)<0.$$
The two prey species cannot coexist without the predator. However, each prey species can coexist with the predator:
$$\lambda_2(\mu_1)<0, \lambda_1(\mu_2)<0, \lambda_3(\mu_1)>0, \lambda_3(\mu_2)>0.$$
As a result of Proposition \ref{prop:2d} there exist unique ergodic measures $\mu_{13}$ and $\mu_{23}$ on the interiors of the positive $x_1x_3$ and $x_2x_3$ planes.
Furthermore, each prey species can invade the stationary system of the other prey species and the predator:
$$\lambda_2(\mu_{13})>0, \lambda_1(\mu_{23})>0.$$
\end{asp}

We note that in this case we cannot use Theorem \ref{t:pers2} because Assumption \ref{a.coexn} does not hold. The goal of this section is to prove persistence in this special case.

\begin{thm}\label{thm2.1}
Suppose that Assumptions \ref{a.nonde} and \ref{a:2p1p} hold.
There exist $\theta$ (see Proposition \ref{prop2.1}),  $n^*\in \N$ (see equation \eqref{e:n*}) and constants $\kappa=\kappa(\theta,T^*)\in(0,1)$ and $K= K(\theta,T^*)>0$ such that
\begin{equation}\label{e:lya}
\E_\bx V^\theta(\BX(n^*T^*))\leq \kappa V^\theta(x)+ K\,\text{ for all }\, \bx\in\R^{3,\circ}_+.
\end{equation}
As a result,
$\BX$ is strongly stochastically persistent. The convergence of the transition probability of $\BX$ in total variation to its unique probability measure $\pi^*$ on $\R^{3,\circ}_+$ is
exponentially fast. Moreover, for any initial value $\mathbf{x}\in\R^{3,\circ}_+$ and any $\pi^*$-integrable function $f$ we have
\begin{equation}\label{slln}
\PP_\bx\left\{\lim\limits_{T\to\infty}\dfrac1T\int_0^Tf\left(\BX^{}(t)\right)dt=\int_{\R_+^{3,\circ}}f(\mathbf{u})\pi^*(d\mathbf{u})\right\}=1.
\end{equation}
\end{thm}

We start with a series of lemmas and propositions.

\begin{lm}\label{lm3.2}
For any invariant probability measure $\pi$ of $\BX$ one has
\begin{equation}\label{lm3.2-e1}
\int_{\R^3_+}\left(\dfrac{\sum c_ix_if_i(\bx)}{1+\bc^\top\bx}-\dfrac{1}2\dfrac{\sum \sigma_{ij}c_ic_jx_ix_jg_i(\bx)g_j(\bx)}{(1+\bc^\top\bx)^2}\right)\pi(d\bx)=0.
\end{equation}
Furthermore,
$$\int_{\R^3_+}\left(\dfrac{x_1f_1(\bx)+x_2f_2(\bx)}{x_1+x_2}-\dfrac{\sum_{i,j=1}^2\sigma_{ij}x_ix_jg_i(\bx)g_j(\bx)}{2(x_1+x_2)^2}\right)\pi(d\bx)=0, \pi\in\{\mu_1,\mu_2,\mu_{13},\mu_{23}\}.$$
\end{lm}
\begin{rmk}
Note that even though $\frac{1}{x_1+x_2}$ is undefined on the set $E_0:=\{(x_1,x_2,x_3)\in \R_+^3~|~x_1+x_2=0\}$ this does not matter since none of the measures $\{\mu_1,\mu_2,\mu_{13},\mu_{23}\}$ put any mass on the set $E_0$.
\end{rmk}
\begin{proof}
We show in \cite[Lemma 3.3]{HN16} that
\begin{equation}\label{lm3.2-e1}
\int_{\R^3_+}\left(\dfrac{\sum c_ix_if_i(\bx)}{1+\bc^\top\bx}-\dfrac{1}2\dfrac{\sum \sigma_{ij}c_ic_jx_ix_jg_i(\bx)g_j(\bx)}{(1+\bc^\top\bx)^2}\right)\pi(d\bx)=0
\end{equation}
for any invariant probability measure $\pi$.
For the second part of the  lemma one can use a contradiction argument similar to \cite[Lemma 3.3 and Lemma 5.1]{HN16}.
\end{proof}
\begin{lm}\label{l:lyapunov}
For any ergodic measure $\mu\in\M$ we have that $\lambda_i(\mu)$ is well defined and finite. Furthermore,
\[
\lambda_i(\mu)=0,~i\in I_\mu.
\]
\end{lm}
\begin{proof}
The proof is the same as the proof of \cite{HN16}[Lemma 5.1].
\end{proof}

We start by proving some general results due to \eqref{a.tight}. In view of \eqref{a.tight},
there is $M>0$
such that
\begin{equation}\label{e2.5}
\left[\dfrac{\sum c_ix_if_i(\bx)}{1+\sum c_ix_i}-\dfrac12\dfrac{\sum \sigma_{ij}c_ic_jx_ix_jg_i(\bx)g_j(\bx)}{(1+\sum c_ix_i)^2}+\gamma_b\left(1+\sum_{i=1}^n (|f_i(\bx)|+g_i^2(\bx))\right)\right]<0
\end{equation}
if $\|\bx\|\geq M$.
Since
\[
|g_i(\bx)g_j(\bx)\sigma_{ij}|\leq 2|\sigma_{ij}|(|g_i(\bx)|^2+|g_j(\bx)|^2)
\]
we can find $\delta_0\in(0,0.5\gamma_b)$ such that
\begin{equation}\label{e2.6}
3\delta_0\sum |g_i(\bx)g_j(\bx)\sigma_{ij}|+\delta_0\sum g_i^2(\bx)\leq \gamma_b\sum g_i^2(\bx)\,,\,\bx\in\R^n_+.
\end{equation}
In view of \eqref{e2.5} and \eqref{e2.6}, we have
\begin{equation}\label{e2.0}
\begin{aligned}
\dfrac{\sum c_ix_if_i(\bx)}{1+\bc^\top\bx}&-\dfrac{1}2\dfrac{\sum \sigma_{ij}c_ic_jx_ix_jg_i(\bx)g_j(\bx)}{(1+\sum c_ix_i)^2}+\gamma_b+\delta_0\sum(2|f_i(\bx)|+g^2_i(\bx))\\
&+3\delta_0\sum |g_i(\bx)g_j(\bx)\sigma_{ij}|<0\,\text{ for all }\, \|\bx\|\geq M.
\end{aligned}
\end{equation}
Using \eqref{e2.0} one can define
\begin{equation}\label{e:H}
\begin{aligned}
H:=\sup\limits_{\bx\in\R^{3}_+}\Bigg\{&\dfrac{\sum c_ix_if_i(\bx)}{1+\bc^\top\bx}-\dfrac{1}2\dfrac{\sum \sigma_{ij}c_ic_jx_ix_jg_i(\bx)g_j(\bx)}{(1+\sum c_ix_i)^2}\\
&+\gamma_b+\delta_0\sum(2|f_i(\bx)|+g_i^2(\bx))
+3\delta_0\sum |g_i(\bx)g_j(\bx)\sigma_{ij}|\Bigg\}<\infty.
\end{aligned}
\end{equation}

\begin{lm}\label{lm3.3}
Suppose the following
\begin{itemize}
\item  The sequences $(\bx_k)_{k\in N}\subset \R_+^3, (T_k)_{k\in \N}\subset \R_+$ are such that $\|\bx_k\|\leq M$, $T_k>1$ for all $k\in \N$ and $\lim_{k\to\infty}T_k=\infty$.

\item The sequence $(\Pi^{\bx_k}_{T_k})_{k\in \N}$ converges weakly to an invariant probability measure
$\pi$.

\item The function $h:\R^3_+\to\R$ is any upper semi-continuous function satisfying
$|h(\bx)|<K_h(1+\bc^\top\bx)^\delta(1+\sum_i(|f_i(\bx)|+|g_i(\bx)|^2))$, $\bx\in \R^n_+$,
for some $K_h\geq 0$, $\delta<\delta_0$.
\end{itemize}
Then one has
\begin{equation}\label{lm3.3-e1}
\lim_{k\to\infty}\int_{\R^n_+}h(\bx)\Pi^{\bx_k}_{T_k}(d\bx)\leq \int_{\R^n_+}h(\bx)\pi(d\bx).
\end{equation}
\end{lm}
\begin{proof}
If the function $h$ is bounded and upper continuous,
\eqref{lm3.3-e1} is obtained from the Portmanteau theorem.
In case $h$ satisfies $|h(\bx)|<K_h(1+\bc^\top\bx)^\delta(1+\sum_i(|f_i(\bx)|+|g_i(\bx)|^2))$, $\bx\in \R^n_+$,
for some $K_h\geq 0$, $\delta<\delta_0$,
we use the uniform bound in \cite[Lemma 3.3]{HN16} the truncated arguments
in \cite[Lemma 3.4]{HN16} to obtain \eqref{lm3.3-e1}.
The details are omitted here.
\end{proof}
It is easy to show that, there exist $p_1, p_2, p_3>0$ such that
\begin{equation}\label{c1e1}
\sum_{i=1}^3p_i\lambda_i(\pi)>0, \pi\in\{\mu_1,\mu_2,\mu_{13},\mu_{23}\}.
\end{equation}

Let $p_0$ be sufficiently large (compared to $p_1,p_2,p_3$) such that
\begin{equation}\label{e2.3}
p_0\min\{\lambda_1(\bdelta^*),\lambda_2(\bdelta^*)\}+\sum_{i=1}^3p_i\lambda_i(\bdelta^*)>0.
\end{equation}

By rescaling $p_0,\dots,p_3$,
we can assume that $\sum_{i=0}^3 p_i\leq\dfrac{\delta_0}4.$
Let
\begin{equation}\label{e2.3}
2\rho^*:=\min\left\{p_0\min\{\lambda_1(\bdelta^*),\lambda_2(\bdelta^*)\}+\sum_{i=1}^3p_i\lambda_i(\bdelta^*), \sum_{i=1}^3p_i\lambda_i(\pi), \pi\in\{\mu_1,\mu_2,\mu_{13},\mu_{23}\}\right\}>0
\end{equation}
and
$P_{\delta}=\left\{\hat \bp:=(\hat p_0,\cdots,\hat p_3)\in\R^4: |\hat p_0|+|\hat p_1|+|\hat p_2|+|\hat p_3|\leq\dfrac{\delta_0}4\right\}$.
For any $\hat \bp$ define the function $V_{\hat\bp}: \R^{3,\circ}_+\to\R_+$ by
\begin{equation}\label{e:V}
V_{\hat \bp}(\bx)=\dfrac{1+\bc^\top\bx}{(x_1+x_2)^{\hat p_0}\prod_{i=1}^3 x_i^{\hat p_i}}.
\end{equation}
Note that if
\[
Z:=\ln  V_{\hat \bp} = \ln(1+\bc^\top\bx) -\hat p_0 \ln(x_1+x_2) -\sum_{i=1}^3 \hat p_i\ln x_i
\]
 then we can write $V_{\hat \bp}^{\delta_0} = e^{\delta_0Z}$. Taking derivatives yields $$\frac{\partial\left( V_{\hat \bp}^{\delta_0}\right)}{\partial x_i}=\delta_0V_{\hat \bp} \frac{\partial Z}{\partial x_i}$$ and
\[
\frac{\partial^2\left( V_{\hat \bp}^{\delta_0}\right)}{\partial x_i\partial x_j}=\delta_0V_{\hat \bp} \left(\delta_0 \frac{\partial Z}{\partial x_i} \frac{\partial Z}{\partial x_j} +  \frac{\partial^2 Z}{\partial x_i \partial x_j}\right).
\]
Using these expressions and the definition of the generator $\Lom$ one can show, after some computations, that
\begin{equation}\label{lv_delta}
\begin{aligned}
\Lom V_{\hat \bp}^{\delta_0}(\bx)=\delta_0V_{\hat \bp}^{\delta_0}(\bx)\bigg[&\dfrac{\sum_{i=1}^3 c_ix_if_i(\bx)}{1+\bc^\top\bx}+\dfrac{\delta_0-1}2\dfrac{\sum_{i,j=1}^3 \sigma_{ij}c_ic_jx_ix_jg_i(\bx)g_j(\bx)}{(1+\bc^\top\bx)^2}\\
&-\sum_{i=1}^3\left(\hat p_if_i(\bx)-\dfrac{\hat p_ig_i^2(\bx)\sigma_{ii}}2\right)\\
&-\hat p_0\left(\dfrac{x_1f_1(\bx)+x_2f_2(\bx)}{x_1+x_2}-\dfrac{\sum_{i,j=1}^2\sigma_{ij}x_ix_jg_i(\bx)g_j(\bx)}{2(x_1+x_2)^2}\right)\\
&+\dfrac{\delta_0}2\hat p_0^2\dfrac{\sum_{i,j=1}^2\sigma_{ij}x_ix_jg_i(\bx)g_j(\bx)}{(x_1+x_2)^2}\\
&+\delta_0\hat p_0\sum_{i=1}^2\sum_{j=1}^3\sigma_{ij}\dfrac{x_ig_i(\bx)}{x_1+x_2}g_j(\bx)-\delta_0\sum_{i=1}^3\dfrac{c_i\hat p_ix_i\sigma_{ij}g_i(\bx)g_j(\bx)}{(1+\bc^\top\bx)}
\\&+\dfrac{\delta_0}2\sum_{i=1}^3 \hat p_i\hat p_j\sigma_{ij}g_i(\bx)g_j(\bx)-\delta_0p_0\sum_{i=1}^2\sum_{j=1}^3\sigma_{ij}\dfrac{x_ic_jx_jg_i(\bx)g_j(\bx)}{(x_1+x_2)(1+\bc^\top\bx)}\bigg].
\end{aligned}
\end{equation}
In virtue of \eqref{e2.0}, we have
\begin{equation}\label{e2.8}
\Lom V_{\hat \bp}^{\delta_0}(\bx)<-\gamma_b\delta_0 V^{ \delta_0}_{\hat\bp}(\bx) \text{ for } \bx\in\R_+^{3,\circ},  \|\bx\|>M, \hat\bp\in P_{\delta_0}.
\end{equation}
Analogously, using \eqref{e:H}
\begin{equation}\label{e2.9}
\Lom V_{\hat\bp}^{ \delta_0}(\bx)<H\delta_0 V^{\delta_0}_{\hat\bp}(\bx),\bx\in\R^{3,\circ}_+, \hat\bp\in P_{\delta_0}.
\end{equation}

Let $\bp=(p_0,\cdots,p_3)$ satisfy \eqref{e2.3} and
consider the function
$$
V(\bx):=V_{ \bp}(\bx)=\dfrac{1+\bc^\top\bx}{(x_1+x_2)^{ p_0}\prod_{i=1}^3 x_i^{ p_i}}.
$$
Let $y_i=\dfrac{x_i}{x_1+x_2}, i=1,2$. Since $y_1+y_2=1$ we have the following estimate
$$
\begin{aligned}
&\dfrac{x_1f_1(\bx)+x_2f_2(\bx)}{x_1+x_2}-\dfrac{\sum_{i,j=1}^2\sigma_{ij}x_ix_jg_i(\bx)g_j(\bx)}{2(x_1+x_2)^2}\\
&\qquad=\sum_{i=1}^2y_i\left(f_i(\bx)-\dfrac{g_i^2(\bx)\sigma_{ii}}2\right) -y_1y_2\sigma_{12}g_1(\bx)g_2(\bx) \\
&\qquad=\sum_{i=1}^2y_i\left(f_i(\bx)-\dfrac{g_i^2(\bx)\sigma_{ii}}2\right)+\dfrac12(y_1+y_2)\left(y_1g_1^2(\bx)\sigma_{11}+y_2g_2^2(\bx)\sigma_{22}\right)\\
&\qquad\qquad+\dfrac12\left(-y_1^2g_1^2(\bx)\sigma_{11}-y_2^2g_2^2(\bx)\sigma_{11}-2y_1y_2\sigma_{12}g_1(\bx)g_2(\bx)\right)\\
&\qquad=y_1\left(f_1(\bx)-\dfrac{g_1^2}2(\bx)\right) + y_2\left(f_2(\bx)-\dfrac{g_2^2}2(\bx)\right)
+\dfrac12y_1y_2\left(g_1^2(\bx)\sigma_{11}+g_2^2(\bx)\sigma_{22}-2g_1(\bx)g_2(\bx)\sigma_{12}\right).
\end{aligned}
$$
Since $(g_i(\bx)g_j(\bx)\sigma_{ij})_{3\times 3}$ is positive definite it is clear that $$g_1^2(\bx)\sigma_{11}+g_2^2(\bx)\sigma_{22}-2\sigma_{12}g_1(\bx)g_2(\bx)\geq 0.$$
Thus, we have
\begin{equation}\label{e2.4}
 \dfrac{x_1f_1(\bx)+x_2f_2(\bx)}{x_1+x_2}-\dfrac{\sum_{i,j=1}^2\sigma_{ij}x_ix_jg_i(\bx)g_j(\bx)}{2(x_1+x_2)^2}
\geq \min\left\{f_1(\bx)-\dfrac{\sigma_{11}g_1^2(\bx)}2, f_2(\bx)-\frac{\sigma_{22}g_2^2(\bx)}2\right\}.
\end{equation}
Define $\Phi:\R_+^{3}\setminus\{(x_1,x_2,x_3)\in\R_+^3~|~x_1+x_2=0\}\mapsto\R$ by
\begin{equation*}
\begin{aligned}
\Phi(\bx)=&\dfrac{\sum_{i=1}^3 c_ix_if_i(\bx)}{1+\bc^\top\bx}-\dfrac{1}2\dfrac{\sum_{i, j=1}^3 \sigma_{ij}c_ic_jx_ix_jg_i(\bx)g_j(\bx)}{(1+\bc^\top\bx)^2}-\sum_{i=1}^3\left(p_if_i(\bx)-\dfrac{p_ig_i^2(\bx)\sigma_{ii}}2\right)\\
&-p_0\left(\dfrac{x_1f_1(\bx)+x_2f_2(\bx)}{x_1+x_2}-\dfrac{\sum_{i,j=1}^2\sigma_{ij}x_ix_jg_i(\bx)g_j(\bx)}{2(x_1+x_2)^2}\right).
\end{aligned}
\end{equation*}
Let $\hat\Phi:\R^3_+\mapsto\R$ be the function
\begin{equation}\label{ehatphi}
\begin{aligned}
\hat\Phi(\bx)=&\dfrac{\sum_{i=1}^3 c_ix_if_i(\bx)}{1+\bc^\top\bx}-\dfrac{1}2\dfrac{\sum_{i,j=1}^3 \sigma_{ij}c_ic_jx_ix_jg_i(\bx)g_j(\bx)}{(1+\bc^\top\bx)^2}-\sum_{i=1}^3\left(p_if_i(\bx)-\dfrac{p_ig_i^2(\bx)\sigma_{ii}}2\right)\\
&-p_0\min\left\{f_1(\bx)-\dfrac{\sigma_{11}g_1^2(\bx)}2, f_2(\bx)-\frac{\sigma_{22}g_2^2(\bx)}2\right\}.
\end{aligned}
\end{equation}
Define $\widetilde\Phi:\R^3_+\mapsto\R$ by
$$
\widetilde\Phi(\bx)=
\begin{cases}
&U(\bx),  \text{ if } ~~~x_1+x_2=0.\\
&\Phi(\bx),  \text{ if } ~~~x_1+x_2\ne 0.
\end{cases}
$$
In view of \eqref{e2.4}, $\widetilde\Phi(\bx)$ is an upper semi-continuous function.

Let $n^*\in\N$ such that
\begin{equation}\label{e:n*}
\gamma_b(n^*-1)>H.
\end{equation}
\begin{lm}\label{lm3.1}
Suppose that Assumptions \ref{a.coexn} and \ref{a.extn2} hold. Let $\bp$ and $\rho^*$ be as in \eqref{e2.3}.
There exists a $T^*>0$ such that, for any $T>T^*$, $\bx\in\partial\R^3_+, \|\bx\|\leq M$ one has
\begin{equation}\label{lm3.1-e1}
\dfrac1T\int_0^T\E_\bx\widetilde\Phi(\BX(t))dt\leq-\rho^*.
\end{equation}
As a corollary,
there is a $\tilde\delta>0$ such that
\begin{equation}\label{lm3.1-e2}\dfrac1T\int_0^T\E_\bx\Phi(\BX(t))dt\leq-\dfrac34\rho^*, T\in[T^*,n^*T^*]
\end{equation}
for any $\bx\in\R^{3,\circ}_+$ satisfying $\|\bx\|\leq M$ and dist$(\bx,\partial\R^3_+)<\tilde\delta.$
\end{lm}
\begin{proof}
We argue by contradiction to obtain \eqref{lm3.1-e1}. Suppose that the conclusion of this lemma is not true.
Then, we can find $\bx_k\in\partial\R^3_+, \|\bx_k\|\leq M$
and $T_k>0$, $\lim_{k\to\infty} T_k=\infty$
such that
\begin{equation}\label{e3.9}
\dfrac1T_k\int_0^{T_k}\E_{\bx_k}\widetilde\Phi(\BX(t))dt>-\rho^*\,,\,k\in\N.
\end{equation}
Define the measures $\Pi^{\bx}_{t}$ by
\[
\Pi^{\bx}_{t}(d\by):=\frac{1}{t}\int_0^t\PP_{\bx}\{\BX(s)\in d\by\}\,ds.
\]
It follows from \cite[Lemma 4.1]{HN16} that $\left(\Pi^{\bx_k}_{T_k}\right)_{k\in \N}$ is tight. As a result
$\left(\Pi^{\bx_k}_{T_k}\right)_{k\in \N}$ has a convergent subsequence in the weak$^*$-topology.
Without loss of generality, we can suppose that $\left(\Pi^{\bx_k}_{T_k}\right)_{k\in \N}$
is a convergent sequence in the weak$^*$-topology.
It can be shown (see Lemma 4.1 from \cite{HN16} or Theorem 9.9 from \cite{EK09}) that its limit is an invariant probability measure $\mu$ of $\BX$. Since $\bx_k\in\partial\R^3_+$, the support of $\mu$ lies in $\partial \R_+^3$.
As a consequence of Lemma \ref{lm3.3}
$$\lim_{k\to\infty}\dfrac1T_k\int_0^{T_k}\E_{\bx_k}\widetilde\Phi(\BX(t))dt\leq\int_{\R^3_+}\widetilde\Phi(\bx)\mu(d\bx).$$
Using Lemmas \ref{lm3.2} and \ref{l:lyapunov},  together with equation \eqref{e2.3} we get that
$$\lim_{k\to\infty}\dfrac1T_k\int_0^{T_k}\E_{\bx_k}\widetilde\Phi(\BX(t))dt\leq -2\rho^*.$$
This contradicts \eqref{e3.9}, which means \eqref{lm3.1-e1} is proved.

With $\hat\Phi$ defined in \eqref{ehatphi}, we have $\hat\Phi(\bx)\geq\Phi(\bx)$ for $x_1+x_2\ne0$ and $\hat\Phi(\bx)=\widetilde\Phi(\bx)$ if $x_1+x_2=0$. As a result of \eqref{e2.3}
Thus
\begin{equation}
 \dfrac1T\int_0^T\E_{(0,0,x_3)}\hat\Phi(\BX(t))dt=
\dfrac1T\int_{0}^T\E_{(0,0,x_3)}
\widetilde\Phi(\BX(t))dt\leq-\rho^*, x_3\leq M, T\geq T^*.
\end{equation}
Due to the Feller property of $(\BX(t))$ on $\R^3_+$ and the continunity of $\hat\Phi$ on $\R^3_+$, there is an $\hat\eps>0$ such that
\begin{equation}\label{lm3.1-e5}
 \dfrac1T\int_0^T\E_{\bx}\hat\Phi(\BX(t))dt\leq -\frac34\rho^*,\text{ if } x_1+x_2\leq\hat\eps, \bx\in\partial\R^3_+,\,\|\bx\|\leq M, T\in[T^*,n^*T^*].
\end{equation}
Together with $\Phi(\bx)\leq\hat\Phi(\bx), x_1+x_2\ne0,$ this implies
$$
 \dfrac1T\int_0^T\E_{\bx}\Phi(\BX(t))dt\leq -\frac34\rho^*,\,\, \bx\in\R^{3,\circ}_+,x_1+x_2\leq\hat\eps, \|\bx\|\leq M, T\in[T^*,n^*T^*].
$$
If $x_1+x_2\ne 0$, then $$\PP_{\bx}\left\{\widetilde\Phi(\BX(t))=\Phi(\BX(t)), t\geq 0 \right\}=1.$$
Using the Feller property of $(\BX(t))$ on $\{(x_1,x_2,x_3)\in\R_+^3~|~x_1+x_2\ne 0\}$, equation \eqref{lm3.1-e1} and the continuity of $\Phi(t)=\wtd\Phi(t)$ on $\{(x_1,x_2,x_3)\in\R_+^3~|~x_1+x_2\ne 0\}$ one can see that there exists $\tilde\delta\in(0,\hat\eps)$ for which
\begin{equation}\label{lm3.1-e6}
 \dfrac1T\int_0^T\E_{\bx}\Phi(\BX(t))dt\leq -\frac34\rho^*, \bx\in\R^{3,\circ}_+, x_1+x_2\geq\hat\eps, \|\bx\|\leq M, \text{dist}(\bx,\partial\R^3_+)<\tilde\delta, T\in[T^*,n^*T^*].
\end{equation}
Combining \eqref{lm3.1-e5} and \eqref{lm3.1-e6} yields \eqref{lm3.1-e2}.

\end{proof}
\begin{lm}\label{lm2.5}
Let $Y$ be a random variable, $\theta_0>0$ a constant, and suppose $$\E \exp(\theta_0 Y)+\E \exp(-\theta_0 Y)\leq K_1.$$
Then the log-Laplace transform
$\phi(\theta)=\ln\E\exp(\theta Y)$
is twice differentiable on $\left[0,\frac{\theta_0}2\right)$ and
$$\dfrac{d\phi}{d\theta}(0)= \E Y,$$
$$0\leq \dfrac{d^2\phi}{d\theta^2}(\theta)\leq K_2\,, \theta\in\left[0,\frac{\theta_0}2\right)$$
 for some $K_2>0$ depending only on $K_1$.
\end{lm}
\begin{proof}
See Lemma 3.5 in \cite{HN16}.
\end{proof}
\begin{prop}\label{prop2.1}
Let $V$ be defined by \eqref{e:V} with $\bp$ and $\rho^*$ satisfying \eqref{e2.3} and $T^*>0$ satisfying the assumptions of Lemma \ref{lm3.1}.
There are $\theta\in\left(0,\frac{\delta_0}2\right)$, $K_\theta>0$, such that for any $T\in[T^*,n^*T^*]$ and $\bx\in\R^{3,\circ}_+, \|\bx\|\leq M$,
$$\E_\bx V^\theta(\BX(T))\leq \exp(-0.5\theta \rho^*T) V^\theta(\bx)+K_\theta.$$
\end{prop}
\begin{proof}
We have from It\^o's formula that
\begin{equation}\label{e:G}
\ln V(\BX(T))=\ln V(\BX(0)) + G(T)
\end{equation}
where
\begin{equation}\label{e:GG}
\begin{aligned}
G(T)=&\int_0^T\Phi(\BX(t))dt+p_0\int_0^T\dfrac{X_1(t)g_1(\BX(t))dE_1(t)+X_2(t)g_2(\BX(t))dE_2(t)}{X_1(t)+X_2(t)}\\
&+ \int_0^T\left[\dfrac{\sum_ic_iX_i(t)g_i(\BX(t))dE_i(t)}{1+\bc^\top\BX(t)}-\sum_ip_ig_i(\BX(t))dE_i(t)\right].
\end{aligned}
\end{equation}
In view of Dynkin's formula, equations \eqref{e:G}, \eqref{e2.9} and Gronwall's inequality
\begin{equation}\label{e3.4_2}
\E_\bx \exp(\delta_0 G(T))=\dfrac{\E_\bx V^{\delta_0}(\BX(T))}{V^{\delta_0}(\bx)}\leq \dfrac{V^{\delta_0}(\bx)+\E_\bx \int_0^t\Lom V^{\delta_0}(\BX(s))\,ds}{V^{\delta_0}(\bx)}\leq  \exp(\delta_0 HT).
\end{equation}
Let $\hat V(\bx):=V_{-\bp}=(1+\bc^\top\bx)(x_1+x_2)^{p_0}\prod_{i=1}^n x_i^{p_1}$.
By virtue of \eqref{e2.9}, we have
\begin{equation}\label{vhat-1}
\dfrac{\E_\bx \hat V^{\delta_0}(\BX(T))}{\hat V^{\delta_0}(\bx)}\leq  \exp(\delta_0 HT).
\end{equation}
Note that
\begin{equation}\label{vhat-2}
V^{-\delta_0}(\bx)=\hat V^{\delta_0}(\bx)(1+\bc^\top\bx)^{-2}\leq \hat V^{\delta_0}(\bx).
\end{equation}
It follows from \eqref{vhat-2} and \eqref{vhat-1} that
\begin{equation}\label{e3.5}
\begin{aligned}
\E_\bx \exp(-\delta_0 G(T))=&\dfrac{\E_\bx V^{-\delta_0}(\BX(T))}{V^{-\delta_0}(\bx)}\\
\leq&\dfrac{\E_\bx\hat V^{\delta_0}(\BX(T))}{V^{-\delta_0}(\bx)}\\
\leq& \dfrac{\E_\bx\hat V^{\delta_0}(\BX(T))}{\hat V^{\delta_0}(\bx)}(1+\bc^\top\bx)^{2}\\
\leq& (1+\bc^\top\bx)^{2}\exp(\delta_0 HT).
\end{aligned}
\end{equation}

By \eqref{e3.4_2} and \eqref{e3.5} the assumptions of Lemma \ref{lm2.5} hold for the random variable $G(T)$. Therefore,
there is $\tilde K_2\geq 0$ such that
\begin{equation}\label{e:phi''}
0\leq \dfrac{d^2\tilde\phi_{\bx,T}}{d\theta^2}(\theta)\leq \tilde K_2\,\text{ for all }\,\theta\in\left[0,\frac{\delta_0}2\right),\, \bx\in\R^{3,\circ}_+, \|\bx\|\leq M, T\in [T^*,n^*T^*]
\end{equation}
where
$$\tilde\phi_{\bx,T}(\theta)=\ln\E_\bx \exp(\theta G(T)).$$

An application of Lemma \ref{lm3.1}, and equation \eqref{e:GG} yields
\begin{equation}\label{e:phi'}
\dfrac{d\tilde\phi_{\bx,T}}{d\theta}(0)=\E_\bx G(T)\leq -\dfrac34\rho^*T
\end{equation}
for all $\bx\in\R^{3,\circ}_+$ satisfying $\|\bx\|\leq M$ and dist$(\bx,\partial\R^3_+)<\tilde\delta.$
By a Taylor expansion around $\theta=0$, for $\|\bx\|\leq M, \dist(\bx,\partial\R^n_+)<\tilde\delta, T\in [T^*,n^*T^*]$ and $\theta\in\left[0,\frac{\delta_0}2\right)$ and using \eqref{e:phi''}-\eqref{e:phi'} we have
$$\tilde\phi_{\bx,T}(\theta)=\tilde\phi_{\bx,T}(0)+\dfrac{d\tilde\phi_{\bx,T}}{d\theta}(0)\theta+ \frac{1}{2} \dfrac{d^2\tilde\phi_{\bx,T}}{d\theta^2}(\xi) (\theta-\xi)^2\leq -\dfrac34\rho^*T\theta+\theta^2\tilde K_2 .$$
If we choose any $\theta\in\left(0,\frac{\delta_0}2\right)$ satisfying
$\theta<\frac{\rho^*T^*}{4\tilde K_2}$, we obtain that
\begin{equation}\label{e3.10}
\tilde\phi_{\bx,T}(\theta)\leq -\dfrac12\rho^*T\theta\,\,\text{ for all }\,\bx\in\R^{3,\circ},\|\bx\|\leq M, \dist(\bx,\partial\R^n_+)<\tilde\delta, T\in [T^*,n^*T^*],
\end{equation}
which leads to
\begin{equation}\label{e3.11}
\dfrac{\E_\bx V^\theta(\BX(T))}{V^\theta(\bx)}=\exp \tilde\phi_{\bx,T}(\theta)\leq\exp(-0.5p^*T\theta).
\end{equation}
In view of \eqref{e2.9},
we have for $\bx$ satisfying $\|\bx\|\leq M, \dist(\bx,\partial\R^n_+)\geq\tilde\delta$ and $T\in  [T^*,n^*T^*]$ that
\begin{equation}\label{e3.12}
\E_\bx V^\theta(\BX(T))\leq \exp(\theta n^*T^*H)\sup\limits_{\|\bx\|\leq M, \dist(\bx,\partial\R^n_+)\geq\tilde\delta}\{V(\bx)\}=:K_\theta<\infty.
\end{equation}
The proof can be finished by combining \eqref{e3.11} and \eqref{e3.12}.
\end{proof}

\begin{proof}[Proof of Theorem \ref{thm2.1}]
Having equation \eqref{e2.8} and Proposition \ref{prop2.1} in hand,
one can mimic the proof of \cite[Theorem 4.1]{HN16}.
\end{proof}
%
%
%
%
%
%
\section{Proof of Theorem \ref{thm:rps}}\label{s:rock-paper-scissors}
Throughout this section we suppose we are in the rock-paper-scissors situation from Definition \ref{deff:rps}.
We note that in this case we cannot use the extinction result \ref{t:ex2} because Assumption \ref{a.extn4} does not hold. Similarly to Lemma \ref{lm3.2},
we can show that
\begin{equation}\label{e.sum}
\int_{\R^3_+}\left(\dfrac{\sum_{i=1}^3 x_if_i(\bx)}{x_1+x_2+x_3}-\dfrac{\sum_{i,j=1}^3\sigma_{ij}x_ix_jg_i(\bx)g_j(\bx)}{2(x_1+x_2+x_3)^2}\right)\pi(d\bx)=0, \text{ for } \pi\in\{\mu_1,\mu_2,\mu_{3}\}.
\end{equation}
\begin{lm}\label{l:ineq}
If $\lambda_2(\mu_1)\lambda_3(\mu_2)\lambda_1(\mu_3)+\lambda_3(\mu_1)\lambda_1(\mu_2)\lambda_2(\mu_3)>0$
then there exist
$p_1, p_2, p_3>0$ such that
\begin{equation}\label{c2e1}
\sum_{i=1}^3p_i\lambda_i(\pi)>0, \pi\in\{\mu_1,\mu_2,\mu_{3}\}.
\end{equation}
If $\lambda_2(\mu_1)\lambda_3(\mu_2)\lambda_1(\mu_3)+\lambda_3(\mu_1)\lambda_1(\mu_2)\lambda_2(\mu_3)<0$
then there exist
$p_1, p_2, p_3>0$ such that
\begin{equation}\label{c2e4}
\sum_{i=1}^3p_i\lambda_i(\pi)<0, \pi\in\{\mu_1,\mu_2,\mu_{3}\}.
\end{equation}
\end{lm}
\subsection{Case 1: $\lambda_2(\mu_1)\lambda_3(\mu_2)\lambda_1(\mu_3)+\lambda_3(\mu_1)\lambda_1(\mu_2)\lambda_2(\mu_3)>0$.}

\begin{thm}\label{thm3.1}
Suppose that Assumption \ref{a.nonde} holds and
\begin{equation}\label{e:per}
\lambda_2(\mu_1)\lambda_3(\mu_2)\lambda_1(\mu_3)+\lambda_3(\mu_1)\lambda_1(\mu_2)\lambda_2(\mu_3)>0.
\end{equation}
Then $\BX(t)$ is strongly stochastically persistent.
\end{thm}
\begin{proof}
	Note that $\lambda_i(\bdelta^*)>0, i=1,2,3$. Combining this property with \eqref{c2e1} implies that
\[
\sum_i p_i \lambda_i(\mu)>0  \text{ for all } \mu \in \Conv(\M).
\]
This shows that Assumption \ref{a.coexn} holds: $\max_{i=1,2,3}\{\lambda_i(\mu)\}>0$ for any $\mu \in \Conv(\M)$. The proof is completed by using Theorem \ref{t:pers2}.
\end{proof}

\subsection{Case 2: $\lambda_2(\mu_1)\lambda_3(\mu_2)\lambda_1(\mu_3)+\lambda_3(\mu_1)\lambda_1(\mu_2)\lambda_2(\mu_3)<0$.}

By Lemma \ref{l:ineq} we can find $p_0,p_1,p_2,p_3>0$ such that $|p_0|+|p_1|+|p_2|+|p_3|<\frac{\delta_0}4$
and
\begin{equation}\label{e.p2}
2\rho^*:=\min\left\{p_0\min_{i\in\{1,2,3\}}\{\lambda_i(\bdelta^*)\}+\sum_{i=1}^3p_i\lambda_i(\bdelta^*); -\sum_{i=1}^3p_i\lambda_i(\pi), \pi\in\{\mu_1,\mu_2,\mu_{3}\}\right\}>0.
\end{equation}
Using the $H$ from \eqref{e:H} define $n_e\in\N$ such that
\begin{equation}\label{e:n_e}
\gamma_b(n_e-1)>H.
\end{equation}
\begin{prop}\label{prop4.1}
Let $U:\R^3_+\setminus\{\0\}$ be defined by
$$U(\bx)=\dfrac{(1+\bc^\top\bx)x_1^{p_1}x_2^{p_2}x_3^{p_3}}{(x_1+x_2+x_3)^{p_0}}$$
 with $\bp$ and $\rho^*$ satisfying \eqref{e.p2}.
There exist constants $T^e>0$, $\theta\in\left(0,\dfrac{\delta_0}2\right)$, $\hat\delta_e>0$, such that for any $T\in[T^e,n^eT^e]$ and $\bx\in\R^{3,\circ}_+, \|\bx\|\leq M$, $\dist(\bx,\partial\R^3_+)<\hat\delta_e$,
$$\E_\bx U^\theta(\BX(T))\leq \exp(-0.5\theta \rho^*T^e) U^\theta(\bx).$$
\end{prop}
\begin{proof}
In view of \eqref{e.sum} and \eqref{e.p2}, this Proposition is proved in the same manner as Proposition \ref{prop2.1}.
\end{proof}

\begin{thm}\label{thm4.1}
Suppose Assumption \ref{a.nonde} holds and
\[
\lambda_2(\mu_1)\lambda_3(\mu_2)\lambda_1(\mu_3)+\lambda_3(\mu_1)\lambda_1(\mu_2)\lambda_2(\mu_3)<0.
\]
For any $\delta<\delta_0$
and any $\bx\in\R^{3,\circ}_+$ we have
\begin{equation}\label{e.extinction}
\lim_{t\to\infty}\E_\bx \bigwedge_{i=1}^3 X_i^{\delta}(t)=0,
\end{equation}
where $\bigwedge_{i=1}^3a_i:=\min_{i=1,\dots,3}\{a_i\}.$
\end{thm}

Before providing a proof of Theorem~\ref{thm4.1}, we provide a sketch of the main ideas.  We wish to prove that the solution starting close enough to the boundary (except the origin) will stay close to the boundary with a large probability under the ``attracting" condition: $$\lambda_2(\mu_1)\lambda_3(\mu_2)\lambda_1(\mu_3)+\lambda_3(\mu_1)\lambda_1(\mu_2)\lambda_2(\mu_3)<0,$$ which says that the absolute value of the product of the negative (attracting) Lyapunov exponents dominates the product of the positive (repelling) Lyapunov exponents.

From Proposition \ref{prop4.1} we get that
$\E_\bx U^\theta(\BX(T))\leq \exp(-0.5\theta \rho^*T^e) U^\theta(\bx)$
when $\|\bx\|$ is not large and $\bx$ is close to the boundary.
When $\|\bx\|$ is large, we have from \eqref{a.tight} that $\Lom U^\theta(\bx)\leq -\theta\gamma_bU^\theta(\bx)$.

The next step is obtaining a Lyapunov-type inequality which will be used for estimating exit times. We show that there $\rho<1$ such that
$$\E_\bx U^\theta(\BX(n_eT_e))\leq \rho U^\theta(\bx),$$ when $U^\theta(\bx)$ is small.

To accomplish this we combine the estimates we have when $\|\bx\|$ is large $(>M)$ and not large $(\leq M)$.
This can be done by analyzing the time  the process $\BX$ hits $\{|\bx|\leq M\}$ (denoted by $\tau$) and the time  $U^\theta(\BX(t))$ exceeds a certain value (denoted by $\xi$).
A few cases are considered and estimated by comparing these stopping times with $(n_e-1)T_e$ where $n_e$ is chosen sufficiently large so that no matter whether $\xi,\tau$ occur before $(n_e-1)T_e$ or after $(n_e-1)T_e$, the process stays for a sufficiently long uninterrupted period of time in either one of the two sets $\{\|\bx\|>M\}$ and $\{\|\bx\|\leq M, \dist(\bx,\partial\R^3_+)\leq \hat\delta\}$ for some small $\hat\delta$.
Then, we can show
$\E_\bx U^\theta(\BX(n_eT_e))\leq \rho U^\theta(\bx)$, $\rho<1$ when $U^\theta(\bx)$ is small.

Once we get this Lyapunov-type inequality, standard arguments from  martingale theory can be used to show that the Markov chain $\{X(kn_eT_e)\}_{k\in\N}$ will stay close to the boundary with a large probability if the initial condition is sufficiently close to the boundary.
This implies that the process $\{\BX(t)\}$ has no invariant measure in the interior of the state space, $\R_+^{3,\circ}$. As a result any weak-limit point of the occupation measures has to be supported on the boundary $\partial \R_+^3$ and \eqref{e.extinction} follows.

\begin{proof}
Similar computations to those showing \eqref{e2.8} yield
\begin{equation}\label{et3.1}
\Lom U^\theta(\bx)\leq -\theta\gamma_bU^\theta(\bx) \text{ if } \|x\|\geq M.
\end{equation}
If we define
$$C_U:=\sup_{\bx\in\R^{3,\circ}_+}\dfrac{(x_1+x_2+x_3)^{p_0}}{1+\bc^\top\bx}$$
then
$$\dist(\bx,\partial\R^3_+)^{p_1+p_2+p_3}=\min\{x_1,x_2,x_3\}^{p_1+p_2+p_3}\leq x_1^{p_1}x_2^{p_2}x_3^{p_3}\leq C_U U(\bx).$$
Let $$\varsigma:=\dfrac{\hat\delta_e^{ (p_1+p_2+p_3)\theta}}{C_U^\theta},$$
and
$$\xi:=\inf\left\{t\geq0: U^\theta(\BX(t))\geq \varsigma\right\}.$$
Clearly, if $U^\theta(\bx)<\varsigma$, then $\xi>0$
for
\begin{equation}\label{e:ine}
\dist(\bx,\partial\R^3_+)\leq \hat\delta\,, t\in [0,\xi).
\end{equation}
If we define $$\tilde U^\theta(\bx):=\varsigma\wedge U^\theta(\bx)$$
we have from the concavity of $x\mapsto x\wedge \varsigma$ that
$$\E_\bx \tilde U^\theta(\BX(T))\leq\varsigma\wedge \E U^\theta(\BX(T)).$$
The stopping time \begin{equation}\label{e:tau}
\tau:=\inf\{t\geq0: \|\BX(t)\|\leq M\}
\end{equation}
combined with \eqref{et3.1} and Dynkin's formula imply that
$$
\begin{aligned}
\E_\bx&\left[ \exp\left(\theta\gamma_b(\tau\wedge\xi\wedge n_eT_e)\right)U^\theta(\BX(\theta\gamma_b(\tau\wedge\xi\wedge n_eT_e))\right]\\
&\leq U^\theta(\bx) +\E_\bx \int_0^{\theta\gamma_b(\tau\wedge\xi\wedge n_eT_e)}\exp(\theta\gamma_b s)[\Lom U^\theta(\BX(s))+ \theta\gamma_bU^\theta(\BX(s))]ds\\
&\leq U^\theta(\bx).
\end{aligned}
$$
As a result,
\begin{equation}\label{et3.3}
\begin{aligned}
U^\theta(x)\geq&
\E_\bx\left[ \exp\left(\theta\gamma_b(\tau\wedge\xi\wedge n_eT_e)\right)U^\theta(\BX(\tau\wedge\xi\wedge n_eT_e))\right]\\
\geq& \E_\bx \left[\1_{\{\tau\wedge\xi\wedge(n_e-1)T_e=\tau\}}U^\theta(\BX(\tau))\right]\\
&+ \E_\bx \left[\1_{\{\tau\wedge\xi\wedge(n_e-1)T_e=\xi\}}U^\theta(\BX(\xi))\right]\\
 &+\exp\left(\theta\gamma_b (n_e-1)T_e\right) \E_\bx \left[\1_{\{(n_e-1)<\tau\wedge\xi<n_eT\}}U^\theta(\BX(\tau\wedge\xi))\right]\\
&+\exp\left(\theta\gamma_b n_eT_e\right) \E_\bx \left[\1_{\{\tau\wedge\xi\geq n_eT_e\}}U^\theta(\BX(n_eT_e))\right].
 \end{aligned}
\end{equation}

By the strong Markov property of $\BX$ and
Proposition \ref{prop4.1} (which we can use because of \eqref{e:ine}), we obtain
\begin{equation}\label{et3.4}
\begin{aligned}
\E_\bx&\left[ \1_{\{\tau\wedge\xi\wedge(n_e-1)T_e=\tau\}}U^\theta(\BX(n_eT_e))\right]\\
&\leq
 \E_\bx \left[\1_{\{\tau\wedge\xi\wedge(n_e-1)T_e=\tau\}}\exp\left(-0.5\theta p_e(n_eT_e-\tau)\right)U^\theta(\BX(\tau\wedge\xi))\right]\\
 &\leq 
 \E_\bx\left[\1_{\{\tau\wedge\xi\wedge(n_e-1)T_e=\tau\}}U^\theta(\BX(\tau\wedge\xi))\right].
 \end{aligned}
\end{equation}
By the strong Markov property of $\BX$ and
Lemma \ref{lm3.3}, we obtain
\begin{equation}\label{et3.5}
\begin{aligned}
\E_\bx&\left[ \1_{\{(n_e-1)T_e<\tau\wedge\xi<n_eT_e\}}U^\theta(\BX(n_eT_e))\right]\\
&\leq
 \E_\bx \left[\1_{\{(n_e-1)T_e<\tau\wedge\xi<n_eT_e\}}\exp\left(\theta H(n_eT_e-\tau)\right)U^\theta(\BX(\tau))\right]\\
 &\leq \exp(\theta HT_e)\E_\bx\left[\1_{\{(n_e-1)T_e<\tau\wedge\xi<n_eT_e\}}U^\theta(\BX(\tau))\right].
 \end{aligned}
\end{equation}
Since we always have $\tilde U^\theta(\BX(n_eT_e))\leq U^\theta(\BX(n_eT_e\wedge\xi))$, we get
\begin{equation}\label{et3.6}
\E_\bx \left[\1_{\{\tau\wedge\xi\wedge(n_e-1)T_e=\xi\}}\tilde U^\theta(\BX(n_eT_e))\right]\leq \E_\bx \left[\1_{\{\tau\wedge\xi\wedge(n_e-1)T_e=\xi\}}U^\theta(\BX(\xi))\right].
\end{equation}
If $U^\theta(\bx)<\varsigma$ then by applying \eqref{et3.4}, \eqref{et3.5} and \eqref{et3.6} to \eqref{et3.3} yields
\begin{equation}\label{et3.8}
\begin{aligned}
\tilde U^\theta(\bx)=U^\theta(\bx)
\geq& \E_\bx \left[\1_{\{\tau\wedge\xi\wedge(n_e-1)T_e=\tau\}}U^\theta(\BX(\tau))\right]\\
&+ \E_\bx \left[\1_{\{\tau\wedge\xi\wedge(n_e-1)T_e=\xi\}}U^\theta(\BX(\xi))\right]\\
 &+\exp\left(\theta\gamma_b (n_e-1)T_e\right) \E_\bx \left[\1_{\{(n_e-1)<\tau\wedge\xi<n_eT\}}U^\theta(\BX(\tau\wedge\xi))\right]\\
&+\exp\left(\theta\gamma_b n_eT_e\right) \E_\bx \left[\1_{\{\tau\wedge\xi\geq n_eT_e\}}U^\theta(\BX(n_eT_e))\right]\\
\geq&\E_\bx \left[\1_{\{\tau\wedge\xi\wedge(n_e-1)T_e=\tau\}}U^\theta(\BX(n_eT_e))\right]\\
&+ \E_\bx \left[\1_{\{\tau\wedge\xi\wedge(n_e-1)T_e=\xi\}}\tilde U^\theta(\BX(n_eT_e))\right]\\
 &+\exp\left(\theta\gamma_b (n_e-1)T_e-\theta HT_e\right)\E_\bx \left[\1_{\{(n_e-1)<\tau\wedge\xi<n_eT\}}U^\theta(\BX(n_eT_e))\right]\\
&+\exp\left(\theta\gamma_b n_eT_e\right) \E_\bx \left[\1_{\{\tau\wedge\xi\geq n_eT_e\}}U^\theta(\BX(n_eT_e))\right]\\
\geq& \E_\bx \tilde U^\theta(\BX(n_eT_e)) \,\quad\text{ (since } \tilde U^\theta(\cdot)\leq U^\theta(\cdot)).
 \end{aligned}
\end{equation}
Clearly, if $U^\theta(\bx)\geq\varsigma$ then
\begin{equation}\label{et3.8a}
\E_\bx \tilde U^\theta(\BX(n_eT_e)) \leq \varsigma=\tilde U^\theta(\bx).
\end{equation}
As a result of \eqref{et3.8}, \eqref{et3.8a} and the Markov property of $\BX$, the sequence
$\{Y(k): k\in\N\}$ where $Y(k):=\tilde U^\theta(\BX(kn_eT_e))$ is a supermartingale.
For $\lambda\leq\varsigma$, let $\tilde\xi_\lambda:=\inf\{k\in\N: Y(k)\geq \lambda\}$.
If $U^\theta(\bx)\leq \lambda\eps$ then we have
\begin{equation}\label{e:EY_ineq}
\E_\bx  Y(k\wedge\tilde\xi_\lambda)\leq\E_\bx Y(0)=U^\theta(\bx)\leq \lambda\eps\,\text{ for all }\, k\in\N.
\end{equation}

We have $\lambda\leq \varsigma$ by assumption and $Y(k)\leq \varsigma$ for any $k$. As a result  \eqref{e:EY_ineq} combined with the Markov inequality yields
$$\PP_\bx\{\tilde\xi_\lambda\leq k\}\leq \lambda^{-1}\E_\bx Y(k\wedge\tilde\xi_\lambda)\leq \eps,
$$
where we used the fact that the event $\{Y(k\vee \wtd\xi_\lambda)\geq  \lambda\}$ is the same as $\{\wtd\xi_\lambda\leq k\}$.
Next, let $k\to\infty$ to get
\begin{equation}\label{et3.7}
\PP_\bx\{\tilde\xi_\lambda<\infty\}\leq \eps.
\end{equation}

Because the solution starting in $\R^{3,\circ}_+$ will remain with probability 1 in $\R^{3,\circ}_+$ and because of the Feller property of $\BX$,
it is not hard to show that for a given compact set $\K\subset\R^{3,\circ}_+$ with nonempty interior, and for any $\eps>0$
there exists a compact set $\wtd\K\subset \R^{3,\circ}_+$ such that
\begin{equation}\label{et3.9}
\PP_\bx\{\BX(t)\in \wtd\K\,\text{ for all }\,t\in[0,n_eT_e]\}>1-\eps,\, \bx\in\K.
\end{equation}

We show by contradiction that $\BX$ is transient.
If the process $\BX$ is recurrent in $\R^{3,\circ}_+$,
then $\BX$ will enter $\K$ in a finite time almost surely given that $\BX(0)\in\R^{3,\circ}_+$.
By the strong Markov property and \eqref{et3.9}, we have
\begin{equation}\label{et3.10}
\PP_\bx\{\BX(kn_eT_e)\in \wtd\K, \text{ for some } k\in\N\}>1-\eps\,,\bx\in\R^{3,\circ}_+.
\end{equation}

Pick a $\lambda\in(0,\varsigma)$ such that $U^\theta(\bx)>\lambda$ for any $\bx\in \wtd\K$.
If  the starting point $\bx$ satisfies $U^\theta(\bx)\leq\lambda\eps$, then  \eqref{et3.7} and \eqref{et3.10} contradict. As a result $\BX$ is transient.

This implies that any weak$^*$-limit of
$P(t,\bx,\cdot)$ is an invariant probability measure with support on $\partial\R^n_+$.
Similar computations to the ones from Lemma \ref{lm3.3} show that if
$P(t_k,\bx_0,\cdot)$ with $\lim_{k\to\infty}t_k=\infty$ converges weakly to $\pi$, and $h(\cdot)$ is a continuous function on $\R^n_+$ such that for all $\bx\in\R_+^n$ we have
$|h(\bx)|<K(1+\|\bx\|)^{\delta},\delta<\delta_0$
then
$\int_{\R^n_+}h(\bx)P(t_k,\bx_0, d\bx)\to \int_{\R^n_+}h(\bx)\pi(d\bx).$

For any
$\pi$ with $\suppo(\pi)\subset\partial\R^n_+$,
we have
$$\int_{\R^n_+}\left(\bigwedge_{i=1}^n x_i\right)^{\delta}\pi(d\bx)=0,$$
and

\[
\left | \left(\bigwedge_{i=1}^n x_i\right)\right |^{\delta}\leq K(1+\|\bx\|)^{\delta}.
\]
These facts imply
$$\lim_{t\to\infty}\int_{\R^n_+}\left(\bigwedge_{i=1}^n x_i\right)^{\delta}P(t_k,\bx_0, d\bx)=0$$
which finishes the proof.
\end{proof}

\begin{lm}\label{lm4.9}
For all $\bx\in\R_+^3$
\begin{equation}\label{lm4.9-e1}
\PP_{\bx}\left\{\lim_{t\to\infty}\dfrac1t\int_0^tg_i(\BX(s))dE_i(s)=0\right\}=1
\end{equation}
and
\begin{equation}\label{lm4.9-e2}
\PP_{\bx}\left\{\lim_{t\to\infty}\dfrac1t\int_0^t\dfrac{\sum X_i(t)g_i(\BX(s))dE_i(s)}{X_1(t)+X_2(t)+X_3(t)}=0\right\}=1.
\end{equation}
\end{lm}
\begin{proof}
Equation \eqref{lm4.9-e1} is proved in \cite[Lemma 5.8]{HN16}.
Equation \eqref{lm4.9-e2} can be proved in the same way using the fact that almost surely
\begin{equation}\label{lm4.9-e3}
\limsup_{t\to\infty}\dfrac{\ln (X_1(t)+X_2(t)+X_3(t))}t\leq 0.
\end{equation}
\end{proof}
\begin{lm}\label{lm4.8}
For any $\bx\in\R_+^{3,\circ}$
$$
\PP_\bx\left\{\limsup_{t\to\infty}\dfrac{\ln \dist(\BX(t),\partial\R^3_+)}{t}
\leq -\dfrac{2\rho^*}{p_1+p_2+p_3}\right\}=1.
$$
\end{lm}
\begin{proof}
First, we show that
for any $\bx\in\R^{3,\circ}_+$,
$$\PP_\bx\Big\{\U(\omega)\subset\Conv(\mu_1,\mu_2,\mu_3)\Big\}=1.$$
 Assume by contradiction that with a positive probability, there is a
(random) sequence $\{t_k\}$ with $\lim_{k\to\infty}t_k=\infty$ such that
 $\wtd \Pi_{t_k}(\cdot)$ converges weakly to an invariant probability of the form
$\pi=(1-\rho)\pi_1+\rho\bdelta^*$
where $\rho\in(0,1]$ and $\pi_1\in\Conv(\mu_1,\mu_2,\mu_3)$.
Define
\[
\Psi(\bx):=\dfrac{\sum_{i=1}^3x_if_i(\bx)}{x_1+x_2+x_3}-\dfrac{\sum_{ij} x_ix_jg_i(\bx)g_j(\bx))\sigma_{ij}}{2(x_1+x_2+x_3)^2}, ~\bx\neq\0
\]
and $\Psi(\0)= \min{\lambda_i(\bdelta)}>0$. One can show, similarly to \eqref{e2.4}, that for all $\bx\neq 0$
\begin{equation*}
 \dfrac{x_1f_1(\bx)+x_2f_2(\bx)+x_3f_3(\bx)}{x_1+x_2+x_3}-\dfrac{\sum_{i,j=1}^3\sigma_{ij}x_ix_jg_i(\bx)g_j(\bx)}{2(x_1+x_2+x_3)^2}
\geq \min_i\left\{f_i(\bx)-\dfrac{\sigma_{ii}g_i^2(\bx)}2\right\}.
\end{equation*}
This together with Lemma \ref{lm4.9} show that with a positive probability
\begin{equation}\label{lm4.8-e2}
\begin{aligned}
\limsup_{k\to\infty}\dfrac1{t_k}&\int_0^{t_k}\left(\dfrac{\sum_{i=1}^3X_if_i(\BX(s))}{X_1(t)+X_2(t)+X_3(t)}-\dfrac{\sum_{ij} X_i(t)X_j(t)g_i(\BX(t))g_j(\BX(t))\sigma_{ij}}{2(X_1(t)+X_2(t)+X_3(t))^2}\right)ds\\
\geq&\int_{\R^3_+}\Psi(\bx)\pi(d\bx)\\
=&(1-\rho)\int_{\R^3_+}\Psi(\bx)\pi_1(d\bx)+\rho\Psi(\0)\\
=&\rho\Psi(\0)>0.
\end{aligned}
\end{equation}
As a result of \eqref{lm4.9-e2}, \eqref{lm4.8-e2} and It\^o's formula we get that with positive probability
$$
\begin{aligned}
\limsup_{k\to\infty}&\dfrac{\ln (X_1(t_k)+X_2(t_k)+X_3(t_k))}t\\
=&\lim_{k\to\infty}\dfrac1{t_k}\int_0^{t_k}\left(\dfrac{\sum_{i=1}^3 X_if_i(\BX(s))}{X_1(t)+X_2(t)+X_3(t)}-\dfrac{\sum_{ij} X_i(t)X_j(t)g_i(\BX(t))g_j(\BX(t))\sigma_{ij}}{2(X_1(t)+X_2(t)+X_3(t))^2}\right)ds\\
>&0
\end{aligned}
$$
which contradicts \eqref{lm4.9-e3}.
As a result of \eqref{lm4.9-e1}, \eqref{e.p2} and It\^o's formula
$$
\begin{aligned}
\limsup_{t\to\infty}\sum_{i=1}^3&p_i\dfrac{\ln X_i(t)}{t}\\
=&\limsup_{t\to\infty}\dfrac1{t}\sum_{i=1}^3 p_i\int_0^{t}\left[\left(f_i(\BX(s))-\dfrac{\sigma_{ii}g_i^2(\BX(s))}2\right)ds+g_i(\BX(s))dE_i(s)\right]\\
=&\limsup_{t\to\infty}\dfrac1t\sum_{i=1}^3 p_i\int_0^t\left(f_i(\BX(s))-\dfrac{\sigma_{ii}g_i^2(\BX(s))}2\right)ds\\
\leq& \sum_{i=1}^3 p_i\sup_{\mu\in\Conv(\mu_1,\mu_2,\mu_3)}\int_{\R^3_+}\left(f_i(\bx)-\dfrac{g_i^2(\bx)}2\right)\mu(d\bx)\leq-2\rho^*.
\end{aligned}
$$
As a result
$$
\limsup_{t\to\infty}\dfrac{\ln \dist(\BX(t),\partial\R^3_+)}{t}
\leq -\dfrac{2\rho^*}{p_1+p_2+p_3}.
$$
This finishes the proof.
\end{proof}
\rps*
\begin{proof}
This follows from Theorems \ref{thm3.1} and \ref{thm4.1}.
\end{proof}

\section{Classification}\label{s:results}
In this section we will list all the possible dynamics (up to permutation) of the stochastic Kolmogorov system \eqref{e:system}. Assumption \ref{a.nonde} is supposed to always hold, and for the extinction results we assume Assumption \ref{a.extn2} holds.

Below, when we will make use of Theorem \ref{t:ex2}, it will be enough to write out what the set of attracting ergodic measures, $\M_1$, is.
 If we say, for example, that $\BX$ converges to $\mu$, what we mean is that $\M_1=\{\mu\}$ and Theorem \ref{t:ex2} holds.
\subsection{All species survive on their own: $\lambda_i(\bdelta^*)>0, i=1,2,3$}
This condition implies that for any $i\in\{1, 2, 3\}$ there exists a unique invariant measure $\mu_i$ with support equal to $\Ri$.
\begin{enumerate}[label=1.\arabic*]
\item All axes are attractors: $\lambda_j(\mu_i)<0$, for $i,j\in\{1,2,3\}, i\ne j$. Then the process converges w.p. 1 to one of the invariant measures $\mu_i, i\in\{1, 2, 3\}$, and with strictly positive probability to $\mu_j$ if $j\in\{1,2,3\}$.
\item Two axes are attractors: $\lambda_j(\mu_i)<0$ for $i\in\{1,2\}$, $j\in\{1,2,3\}, i\ne j$. If $\max_{i}\{\lambda_i(\mu_3)\}>0$ then the process converges w.p. 1 to one of the invariant measures $\mu_i,  i\in\{1,2\}$, and with strictly positive probability to $\mu_j$ if $j\in\{1, 2\}$
\item One axis is an attractor: $\lambda_i(\mu_1)<0$ for $i\in\{2,3\}$,  $\lambda_3(\mu_2)>0$,  $\lambda_2(\mu_3)  >0$. There exists an invariant measure $\mu_{23}$ on $\R_{23+}^\circ$.
If $\lambda_1(\mu_{23})>0$, the process converges to $\mu_1$. If $\lambda_1(\mu_{23})<0$, the process converge either to $\mu_1$ or $\mu_{23}$.
\item One axis is an attractor: $\lambda_i(\mu_1)<0$ for $i\in\{2,3\}$,  $\lambda_3(\mu_2)>0$, $\lambda_2(\mu_3)<0$, $\lambda_1(\mu_3) >0$. Then the process converges to $\mu_1$.
\item One axis is an attractor: $\lambda_i(\mu_1)<0$ for $i\in\{2,3\}$, $\lambda_1(\mu_2)>0$, $\lambda_3(\mu_2)<0$, $\max\{\lambda_1(\mu_3),\lambda_2(\mu_3)\}>0$. The process converges to $\mu_1$.

\item No axis is an attractor, no face has an invariant measure (Rock-Paper-Scissors): $\lambda_2(\mu_1)>0$, $\lambda_3(\mu_1)<0$, $\lambda_3(\mu_2)>0$, $\lambda_1(\mu_2)<0$, $\lambda_1(\mu_3)>0,\lambda_2(\mu_3)<0$. If $|\lambda_2(\mu_1)\lambda_3(\mu_2)\lambda_1(\mu_3)|>|\lambda_3(\mu_1)\lambda_1(\mu_2)\lambda_2(\mu_3)|$ we get persistence. If $|\lambda_2(\mu_1)\lambda_3(\mu_2)\lambda_1(\mu_3)|<|\lambda_3(\mu_1)\lambda_1(\mu_2)\lambda_2(\mu_3)|$ then we get extinction in the following sense:
    \[
    \lim_{t\to\infty}\E_\bx \left(\bigwedge_{i=1}^3 X_i(t)\right)^{\delta}=0, \text{for all small enough}~\delta.
    \]
    Furthermore, there exists $\alpha>0$ such that with probability $1$
\[
\limsup_{t\to\infty} \frac{\ln \left(\dist(\BX(t),\partial\R^3_+)\right)}{t}<-\alpha.
\]
\item No axis is an attractor, one face has an invariant measure: $\lambda_2(\mu_1)>0$, $\lambda_3(\mu_1)<0$, $\lambda_1(\mu_2)>0$, $\lambda_3(\mu_2)<0$ and $\max\{\lambda_1(\mu_3),\lambda_2(\mu_3)\}>0$. There exists $\mu_{12}$.
If $\lambda_3(\mu_{12})>0$, the system is persistent. \textit{Assumption \ref{a:hof} can be seen to hold as follows: Suppose $\lambda_1(\mu_3)>0.$ Then let $p_2=1$ and pick $p_3>0$ small enough such that $\lambda_2(\mu_1)+p_3\lambda_3(\mu_1)>0$. Finally, since $\lambda_1(\mu_2), \lambda_1(\mu_3) >0$ we can pick $p_1>0$ large enough such that $p_1\lambda_1(\mu_2)+ p_3\lambda_3(\mu_2)>0$ and $p_1\lambda_1(\mu_3)+ \lambda_2(\mu_3)>0$.}

If $\lambda_3(\mu_{12})<0$, the process converges to $\mu_{12}$.
\item No axis is an attractor, one face has an invariant measure: $\lambda_2(\mu_1)>0$, $\lambda_3(\mu_1)<0$, $\lambda_i(\mu_2)>0$, $i\in\{1,3\}$, $\lambda_1(\mu_3)>0, \lambda_2(\mu_3)<0$. There exists $\mu_{12}$.
If $\lambda_3(\mu_{12})>0$, the system is persistent. \textit{Assumption \ref{a:hof} can be seen to hold as follows: Let $p_2=1$ and pick $p_3>0$ small enough such that $\lambda_2(\mu_1)+p_3\lambda_3(\mu_1)>0$. Then pick $p_1>0$ large enough such that $p_1\lambda_1(\mu_3)+ \lambda_2(\mu_3)>0$. }

If $\lambda_3(\mu_{12})<0$,
the process converges to $\mu_{12}$.
\item No axis is an attractor, two faces have invariant measures: $\lambda_2(\mu_1)>0$, $\lambda_3(\mu_1)<0$, $\lambda_i(\mu_2)>0$, $i\in\{1,3\}$, $\lambda_2(\mu_3)>0$. The invariant measures $\mu_{12},\mu_{23}$ exist.
If $\lambda_3(\mu_{12})>0, \lambda_1(\mu_{23})>0$, the system is persistent.
\textit{Assumption \ref{a:hof} can be seen to hold as follows: Let $p_2=1$ and pick $p_3>0$ small enough such that $\lambda_2(\mu_1)+p_3\lambda_3(\mu_1)>0$. Then pick $p_1>0$ small enough such that $p_1\lambda_1(\mu_3)+ \lambda_2(\mu_3)>0$. }

If $\lambda_3(\mu_{12})<0, \lambda_1(\mu_{23})>0$,
the process converges to $\mu_{12}$.
If $\lambda_3(\mu_{12})>0, \lambda_1(\mu_{23})<0$,
the process converges to $\mu_{23}$.
If $\lambda_3(\mu_{12})<0, \lambda_1(\mu_{23})<0$,
the process converges to either $\mu_{12}$ or $\mu_{23}$.
\item No axis is an attractor, all faces have invariant measures: $\lambda_j(\mu_i)>0$, for $i,j\in\{1,2,3\}, i\ne j$. The invariant measures $\mu_{12},\mu_{13},\mu_{23}$ exist.
Without loss of generality, suppose $\lambda_3(\mu_{12})\leq\lambda_3(\mu_{23})\leq\lambda_1(\mu_{23})$
If they are all positive, the system is persistent.
\textit{Assumption \ref{a:hof} can be seen to hold as follows: Pick any $p_1, p_2, p_3>0.$ }

If they are all negative, with probability $1$ the process converges to one of them.
If $\lambda_3(\mu_{12})<0<\lambda_3(\mu_{23})\leq\lambda_1(\mu_{23})$,
the process converges to $\mu_{12}$.
If $\lambda_3(\mu_{12})\leq\lambda_3(\mu_{23})<0<\lambda_1(\mu_{23})$,
the process converges to $\mu_{12}$ or $\mu_{23}$.
\end{enumerate}

\subsection{Two species survive on their own: $\lambda_i(\bdelta^*)>0, i=1,2$ and $\lambda_3(\bdelta^*)<0$}
For any $i\in\{1, 2\}$ there exists a unique invariant measure $\mu_i$ with support equal to $\Ri$.
\begin{enumerate}[label=2.\arabic*]
\item Two axes  are attractors $\lambda_j(\mu_i)<0$, for $i\in\{1,2\}, j\in\{1,2,3\}\setminus\{i\}$. Then the process converges w.p. 1 to one of the invariant measures $\mu_i, i\in\{1, 2\}$, and with strictly positive probability to $\mu_j$ if $j\in\{1,2\}$.
\item One axis is an attractor, no face has an invariant measure: $\lambda_2(\mu_1)<0, \lambda_3(\mu_1)<0$ and $\lambda_1(\mu_2)>0, \lambda_3(\mu_2)<0$. Then the process converges to $\mu_1$.
\item One axis is an attractor, one face has an invariant measure: $\lambda_2(\mu_1)<0, \lambda_3(\mu_1)<0$ and $\lambda_3(\mu_2)>0$. Then $\mu_{23}$ exists.
     \begin{enumerate}
       \item  If $\lambda_1(\mu_{23})>0$ the process converges to $\mu_1$.
       \item If $\lambda_1(\mu_{23})<0$ the process converges to $\mu_1$ or to $\mu_{23}$.
     \end{enumerate}

\item No axis is an attractor, only face $12$ has an invariant measure: $\lambda_3(\mu_1)<0, \lambda_3(\mu_2)<0, \lambda_2(\mu_1)>0, \lambda_1(\mu_2)>0$. Then $\mu_{12}$ exists.
    \begin{enumerate}
      \item  If $\lambda_3(\mu_{12})<0$ the process converges to $\mu_{12}$.
      \item If $\lambda_3(\mu_{12})>0$ there is persistence. \textit{Assumption \ref{a:hof} can be seen to hold as follows: Let $p_3=1$ and pick $p_1>0$ large enough such that $p_1\lambda_1(\mu_2)+p_3\lambda_3(\mu_2)>0$. Then pick $p_2>0$ large enough such that $p_2\lambda_2(\mu_1)+ \lambda_3(\mu_1)>0$ and $p_1\lambda_1(\bdelta^*) + p_2\lambda_2(\bdelta^*)+ \lambda_3(\bdelta^*)>0$. }
    \end{enumerate}

\item No axis is an attractor, only face $13$ has an invariant measure: $\lambda_3(\mu_1)>0, \lambda_3(\mu_2)<0$, $\lambda_1(\mu_2)>0, \lambda_2(\mu_1)<0$. Then $\mu_{13}$ exists.
    \begin{enumerate}
      \item If $\lambda_2(\mu_{13})<0$ the process converges to $\mu_{13}$.
      \item If $\lambda_2(\mu_{13})>0$ there is persistence. \textit{Assumption \ref{a:hof} can be seen to hold as in the previous case with the roles of the indices $2$ and $3$ interchanged.}
    \end{enumerate}

\item No axis is an attractor, only faces $13$ and $12$ have an invariant measure: $\lambda_3(\mu_1)>0, \lambda_3(\mu_2)<0$, $\lambda_1(\mu_2)>0, \lambda_2(\mu_1)>0$. Then $\mu_{13}, \mu_{12}$ exist.
     \begin{enumerate}
       \item  If $\lambda_2(\mu_{13})>0, \lambda_3(\mu_{12})>0$ there is persistence. \textit{Assumption \ref{a:hof} can be seen to hold as follows: Let $p_2=p_3=1$ and pick $p_1>0$ large enough such that $p_1\lambda_1(\mu_2)+\lambda_3(\mu_2)>0$ and  $p_1\lambda_1(\bdelta^*) + \lambda_2(\bdelta^*)+ \lambda_3(\bdelta^*)>0$. }
       \item If $\lambda_2(\mu_{13})<0, \lambda_3(\mu_{12})>0$ the process converges to $\mu_{13}$.
       \item  If $\lambda_2(\mu_{13})>0, \lambda_3(\mu_{12})<0$ the process converges to $\mu_{12}$.
       \item   If $\lambda_2(\mu_{13})<0, \lambda_3(\mu_{12})<0$ the process converges w.p. 1 to $\mu_{12}$ or $\mu_{13}$.
     \end{enumerate}

\item No axis is an attractor, only faces $13$ and $23$ have an invariant measure: $\lambda_3(\mu_1)>0, \lambda_3(\mu_2)>0$, $\min\{\lambda_1(\mu_2), \lambda_2(\mu_1)\}<0$. Then $\mu_{13}, \mu_{23}$ exist.
\begin{itemize}
  \item Say $\lambda_1(\mu_2)<0, \lambda_2(\mu_1)>0$.
  \begin{enumerate}
    \item If $\lambda_2(\mu_{13})>0, \lambda_1(\mu_{23})>0$ there is persistence. \textit{Assumption \ref{a:hof} can be seen to hold as follows: Let $p_3=1$ and pick $p_1>0$ small enough such that $p_1\lambda_1(\mu_2)+\lambda_3(\mu_2)>0$. Then pick $p_2>0$ large enough such that $p_1\lambda_1(\bdelta^*) + p_2\lambda_2(\bdelta^*)+ \lambda_3(\bdelta^*)>0$. }
    \item  If $\lambda_2(\mu_{13})<0, \lambda_1(\mu_{23})>0$ the process converges to $\mu_{13}$.
    \item If $\lambda_2(\mu_{13})>0, \lambda_1(\mu_{23})<0$ the process converges to $\mu_{23}$.
    \item  If $\lambda_2(\mu_{13})<0, \lambda_1(\mu_{23})<0$ the process converges w.p. 1 to $\mu_{13}$ or $\mu_{23}$.
  \end{enumerate}

  \item Say $\lambda_1(\mu_2)<0, \lambda_2(\mu_1)<0$.
  \begin{enumerate}
    \item If $\lambda_2(\mu_{13})>0, \lambda_1(\mu_{23})>0$ there is persistence (this special case is treated in Section \ref{s:predprey}).
    \item If $\lambda_2(\mu_{13})<0, \lambda_1(\mu_{23})>0$ the process converges to $\mu_{13}$.
    \item If $\lambda_2(\mu_{13})>0, \lambda_1(\mu_{23})<0$ the process converges to $\mu_{23}$.
    \item If $\lambda_2(\mu_{13})<0, \lambda_1(\mu_{23})<0$ the process converges w.p. 1 to $\mu_{13}$ or $\mu_{23}$.
  \end{enumerate}

\end{itemize}
\item No axis is an attractor, all faces have an invariant measure: $\lambda_3(\mu_1)>0, \lambda_3(\mu_2)>0$, $\lambda_1(\mu_2)>0, \lambda_2(\mu_1)>0$. Then $\mu_{13}, \mu_{12}, \mu_{23}$ exist.
    \begin{enumerate}
      \item  If $\lambda_1(\mu_{23})>0, \lambda_2(\mu_{13})>0, \lambda_3(\mu_{12})>0$ there is persistence. \textit{Assumption \ref{a:hof} can be seen to hold as follows: Let $p_2=p_3=1$ and pick $p_1>0$ large enough such that $p_1\lambda_1(\bdelta^*) + \lambda_2(\bdelta^*)+ \lambda_3(\bdelta^*)>0$. }
      \item If $\lambda_1(\mu_{23})<0, \lambda_2(\mu_{13})>0, \lambda_3(\mu_{12})>0$ the process converges to $\mu_{23}$.
      \item If $\lambda_1(\mu_{23})>0, \lambda_2(\mu_{13})<0, \lambda_3(\mu_{12})>0$ the process converges to $\mu_{13}$.
      \item  If $\lambda_1(\mu_{23})>0, \lambda_2(\mu_{13})>0, \lambda_3(\mu_{12})<0$ the process converges to $\mu_{12}$.
      \item  If $\lambda_1(\mu_{23})<0, \lambda_2(\mu_{13})<0, \lambda_3(\mu_{12})>0$ the process converges w.p. 1 to $\mu_{23}$ or $\mu_{13}$.
      \item If $\lambda_1(\mu_{23})<0, \lambda_2(\mu_{13})>0, \lambda_3(\mu_{12})<0$ the process converges w.p. 1 to $\mu_{23}$ or $\mu_{12}$.
      \item If $\lambda_1(\mu_{23})>0, \lambda_2(\mu_{13})<0, \lambda_3(\mu_{12})<0$ the process converges w.p. 1 to $\mu_{13}$ or $\mu_{12}$.
      \item If $\lambda_1(\mu_{23})<0, \lambda_2(\mu_{13})<0, \lambda_3(\mu_{12})<0$ the process converges w.p. 1 to $\mu_{12}, \mu_{23}$ or $\mu_{13}$.
    \end{enumerate}
\end{enumerate}

\subsection{One species survives on its own: $\lambda_1(\bdelta^*)>0$ and $\lambda_i(\bdelta^*)<0$, i=2,3}
The condition $\lambda_1(\bdelta^*)>0$ implies that there exists a unique invariant measure $\mu_i$ with support equal to $\Rx$.
\begin{enumerate}[label=3.\arabic*]
\item One axis is an attractor: $\lambda_2(\mu_1)<0, \lambda_3(\mu_1)<0$. Then the process converges to $\mu_1$.
\item No axis is an attractor, one face has an invariant measure: $\lambda_2(\mu_1)>0, \lambda_3(\mu_1)<0$. Then $\mu_{12}$ exists.
\begin{enumerate}
  \item If $\lambda_3(\mu_{12})>0$ there is persistence. \textit{Assumption \ref{a:hof} can be seen to hold as follows: Let $p_3=1$ and pick $p_2>0$ large enough such that  $p_2\lambda_2(\mu_1) + \lambda_3(\mu_1)>0$. Next, pick $p_1$ large enough such that
      $p_1\lambda_1(\bdelta^*) + p_2\lambda_2(\bdelta^*)+ \lambda_3(\bdelta^*)>0$. }
  \item If $\lambda_3(\mu_{12})<0$ the process converges to $\mu_{12}$.
\end{enumerate}

\item No axis is an attractor, two faces have invariant measures: $\lambda_2(\mu_1)>0, \lambda_3(\mu_1)>0$. Then $\mu_{12}, \mu_{13}$ exist.
\begin{enumerate}
  \item If $\lambda_3(\mu_{12})>0, \lambda_2(\mu_{13})>0$ there is persistence. \textit{Assumption \ref{a:hof} can be seen to hold as follows: Let $p_2=p_3=1$ and pick $p_1>0$ large enough such that $p_1\lambda_1(\bdelta^*) + \lambda_2(\bdelta^*)+ \lambda_3(\bdelta^*)>0$. }
  \item If $\lambda_3(\mu_{12})<0, \lambda_2(\mu_{13})>0$ the process converges to $\mu_{12}$.
  \item If $\lambda_3(\mu_{12})>0, \lambda_2(\mu_{13})<0$ the process converges to $\mu_{13}$.
  \item  If $\lambda_3(\mu_{12})<0, \lambda_2(\mu_{13})<0$ the process converges w.p. 1 to $\mu_{12}$ or $\mu_{13}$.
\end{enumerate}

\end{enumerate}

\section{Applications}\label{s:LV}
Our main results concern the classification of the possible asymptotic outcomes of three-dimensional Kolmogorov systems. In this section, we first show how for many $3$-dimensional Lotka--Volterra systems that our assumptions, and therefore our results, hold. In particular, we  prove that the Lyapunov exponents can be computed explicitly by solving a system of linear equations. Second, we give an example of a modified Lotka-Volterra system where the conditions for stochastic persistence are less restrictive than the conditions for permanence of the corresponding deterministic model.

\subsection{Lotka-Volterra Systems}
For the Lotka-Volterra systems, we assume the dynamics are given by the stochastic differential equations
\begin{equation}\label{e:LV}
dX_i(t)=X_i(t)\left(m_i+\sum_{j=1}^3 a_{ij} X_j(t)\right)\,dt + X_i(t)\,dE_i(t), X_i(0)=x_i\geq 0.
\end{equation}
The constant $m_i$ is the per-capita growth rate of species $i$, and $a_{ij}$ is the coefficient measuring the per-capita interaction strength of species $j$ on species $i$.

We assume that each species experiences intraspecific competition and there are no mutualistic interactions, which even for the deterministic Lotka-Volterra equations can lead to finite-time blow up of solutions.

\begin{asp}\label{asp:LV}
For the Lotka-Volterra system~\eqref{e:LV}, assume that  $a_{ii}<0$ for all $i$, and $a_{ij}>0$ for $i\neq j$ implies $a_{ji}<0$.
\end{asp}

The following is a proposition verifying \eqref{a.tight} of Assumption~\ref{a.nonde}. The rest of Assumption~\ref{a.nonde} as well as Assumption~\ref{a.extn2} follow immediately.

\begin{prop}\label{p:tight} If Assumption~\ref{asp:LV} holds, then \eqref{e:LV} satisfies \eqref{a.tight}.
\end{prop}

Next we show that the external Lyapunov exponents can be found by solving a system of linear equations.

\begin{prop}\label{prop:LV} Assume \eqref{e:LV} satisfies Assumption~\ref{asp:LV}. Let $\mu$ be an ergodic invariant probability measure for \eqref{e:LV}. If there exists a unique solution $\bar \bx$ to the system of linear equations
\begin{equation}\label{e:sys}
\begin{aligned}
m_i+\sum_j a_{ij}\bar x_j -\frac{\sigma_{ii}}{2}=0 &\mbox{ for }i\in I_\mu\\
\bar x_i=0& \mbox{ for }i\notin I_\mu,
\end{aligned}
\end{equation}
then
\[
\lambda_i(\mu)=m_i+\sum_j a_{ij}\bar x_j -\frac{\sigma_{ii}}{2}
\mbox{ for all }i.\]

\end{prop}

\begin{rmk} Using Proposition~\ref{prop:1d} and Remark~\ref{rmk:2d}, one easily show inductively on the cardinality $|I_\mu|=0,1,2$ that non-zero external Lyapunov exponents imply that \eqref{e:sys} has a unique solution i.e. the coefficient matrix $\{a_{ij}\}_{i,j\in I_\mu}$ restricted to the supported species is invertible.
\end{rmk}

To illustrate the applicability of our results to a specific model we consider a model of rock-paper-scissors and contrast  the difference between the deterministic and stochastic dynamics. To this end, pick $0<\beta<1<\alpha$ and consider the following system of differential equations:
  \begin{equation}\label{e:ODE}
\begin{split}
d\bar X_1(t) &= \bar X_1(t)\left(1-\bar X_1(t)-\alpha \bar X_2(t)-\beta \bar X_3(t)\right)\,dt \\
d\bar X_2(t) &= \bar X_2(t)\left(1-\beta X_1(t)-\bar X_2(t)-\alpha\bar X_3(t)\right)\,dt \\
d\bar X_3(t) &= \bar X_3(t)\left(1-\alpha \bar X_1(t) -\beta \bar X_2(t) - \bar X_3(t)\right)\,dt. \\
\end{split}
\end{equation}
This is the model introduced by \cite{ML75}. One can see that \eqref{e:ODE} has five fixed points. The origin $0$ is a source, the canonical basis vectors $e_1, e_2, e_3$ are saddle points and the interior equilibrium is given by
\[
\bar x = \left(\frac{1}{1+\alpha+\beta},\frac{1}{1+\alpha+\beta},\frac{1}{1+\alpha+\beta}\right).
\]
Let $D=\{\bx\in \R^3_+:x_1=x_2=x_3\}$ and $\Delta=\{\bx\in\R^3_+:\sum_ix_i=1\}.$ For these equations, the equilibria $e_i$ and the connecting orbits (i.e. the unstable manifolds) form a heteroclinic cycle $\Omega$. \cite{HJ89} provide the following classification of the dynamics:
\begin{enumerate}
  \item If $\alpha+\beta<2$ the interior equilibirium $\bar x$ is globally stable and all trajectories starting in $\R_+^{3,\circ}$ converge to $\bar x$.
  \item If $\alpha+\beta>2$ the interior equilibrium $\bar x$ is a saddle with stable manifold $D\setminus\{0\}$. Every trajectory starting from $\R_+^{3,\circ}\setminus D$ has $\Omega$ as its $\omega$-limit set.
  \item If $\alpha+\beta=2$ the set $\Delta$ is invariant and attracts all nonzero trajectories, $\Omega=\partial \Delta$ and trajectories starting in $\Delta^\circ\setminus \{\bar x\}$ are periodic.
\end{enumerate}

A stochastic counterpart to these equations is given by
\begin{equation}\label{e:SDE2}
\begin{split}
dX_1(t) &= X_1(t)\left(1-X_1(t)-\alpha X_2(t)-\beta X_3(t)\right)\,dt +  X_1(t) \,dE_1(t)\\
dX_2(t) &= X_2(t)\left(1-\beta X_1(t)-X_2(t)-\alpha X_3(t)\right)\,dt +  X_2(t)\,dE_2(t)\\
dX_3(t) &= X_3(t)\left(1-\alpha X_1(t) -\beta X_2(t) - X_3(t)\right)\,dt + X_3(t) \,dE_3(t)\\
\end{split}
\end{equation}
with $\Sigma = \diag(\sigma, \sigma, \sigma)$.

Using Theorem~\ref{thm:rps} we can prove the following proposition.

\begin{prop}\label{p:rps} If $\sigma<2$, then there is the following dichotomy:
\begin{enumerate}
  \item If $\alpha+\beta<2$ the species persist and the system converges to a unique invariant probability measure on $\R_+^{3,\circ}$.
  \item If $\alpha+\beta>2$ there is extinction, in the sense that for all starting points we have with probability one that $$\BX(t)\to \partial \R_+^3.$$
\end{enumerate}
\end{prop}

\begin{rmk}
System~\eqref{e:SDE2} is an example of a competitive, Lotka-Volterra SDE i.e. the intrinsic rates of growth $m_i$ are positive, the interspecific interaction coefficients $a_{ij}$ are non-positive for $i\neq j$, and the intraspecific interaction coefficients $a_{ii}$ are negative. For these competitive, Lotka-Volterra SDE, the results of \citet{Z93} can be used to show that these SDE can for appropriate parameter choices exhibit all of the dynamics shown in Figure~\ref{fig:one} except for type (viii) i.e. one can not have positive probability of asymptotically approaching each of the species pairs.
\end{rmk}

\subsection{Stochastic Persistence Despite Deterministic Impermanence}

In the deterministic literature, permanence is the deterministic analog of stochastic persistence. However, as we shall show, there are cases where a deterministic system is not permanent but the corresponding stochastic system is strongly stochastically persistent. To this end, we consider a modified Lotka-Volterra model of two competing species that share a predator. The modification comes from assuming that the predator exhibits a switching functional response whereby the predator spends more time searching for the more common prey species. In this model, $X_1,X_2$ denote the prey densities, and $X_3$ the predator density. The equations of motion for the deterministic model are
\begin{equation}\label{e:switch:ode}
\begin{aligned}
dX_1(t)=&X_1(t)\left(r-X_1(t)-\beta X_2(t)-\frac{X_1(t)}{X_1(t)+X_2(t)}X_3(t)  \right)dt\\
dX_2(t)=&X_2(t)\left(r-X_2(t)-\beta X_1(t)-\frac{X_2(t)}{X_1(t)+X_2(t)}X_3(t)  \right)dt\\
dX_3(t)=&X_3(t)\left(\frac{X_1(t)^2+X_2(t)^2}{X_1(t)+X_2(t)}-d-cX_3(t)  \right)dt
\end{aligned}
\end{equation}
where $r>0$ is the intrinsic rate of growth of the prey species, $\beta>0$ is the strength of intraspecific competition,  $d$ is the density-independent predator death rate, and $c$ is the strength of intraspecific competition for the predator. The term $X_i(t)/(X_1(t)+X_2(t))$ represents the probability that a predator is searching for prey $i$ i.e. a predator is more likely to search for the more common prey. The system of ODEs~\eqref{e:switch:ode} is nearly the same as those considered by \citet{teramoto_kawasaki1979,hutson1984}; they only differ by the inclusion of a self-limitation term in the predator.

A key concept of coexistence in the mathematical ecology literature is permanence~\citep{H81,H84,HS98,schreiber2000,patel_schreiber2017} in which asymptotically all species densities are uniformly bounded above and away from zero for all positive initial conditions.

\begin{deff} The system of differential equations~\eqref{e:switch:ode} is \emph{permanent} if there exists $m>0$ such that
\[
\frac{1}{m}\le \liminf_{t\to\infty} \min_i X_i(t)\le \limsup_{t\to\infty} \max_i X_i(t)\le m
\]
whenever $\min_i X_i(0)>0.$
\end{deff}

The following proposition characterizes, generically, when \eqref{e:switch:ode} is permanent or not permanent, i.e., impermanent.

\begin{prop}\label{prop:switch:ode} Assume $\beta>1$. If
\begin{equation}\label{e:switch:permanent}
\frac{r}{1+\beta}>d \mbox{ and }\frac{r}{\beta}(1+c(1-\beta))>d
\end{equation}
then \eqref{e:switch:ode} is permanent. If either inequality of \eqref{e:switch:permanent} is reversed, then \eqref{e:switch:ode} is not permanent.
\end{prop}

Next, we consider the SDE analog of \eqref{e:switch:ode}:
\begin{equation}\label{e:switch:sde}
\begin{aligned}
dX_1(t)=&X_1(t)\left(r-X_1(t)-\beta X_2(t)-\frac{X_1(t)}{X_1(t)+X_2(t)}X_3(t)  \right)dt+\varepsilon X_1(t)dB_1(t)\\
dX_2(t)=&X_2(t)\left(r-X_2(t)-\beta X_1(t)-\frac{X_2(t)}{X_1(t)+X_2(t)}X_3(t)  \right)+\varepsilon X_2(t)dB_2(t)dt\\
dX_3(t)=&X_3(t)\left(\frac{X_1(t)^2+X_2(t)^2}{X_1(t)+X_2(t)}-d-cX_3(t)  \right)dt+\varepsilon X_3(t)dB_3(t)
\end{aligned}
\end{equation}
where $B_1(t),B_2(t),B_3(t)$ are independent, standard Brownian motions i.e. $\rm{Var}(B_i(t))=t.$ For this model, our results yield the following proposition about strong, stochastic persistence.

\begin{prop}\label{prop:switch:sde} Assume $\beta>0$. If
\begin{equation}
\frac{r}{\beta}(1+c(1-\beta))>d \mbox{ and }\varepsilon>0\mbox{ is sufficiently small,}
\end{equation}
then \eqref{e:switch:sde} is strongly, stochastically persistent.
\end{prop}

\begin{rmk}
Propositions~\ref{prop:switch:ode} and \ref{prop:switch:sde} imply that for $\frac{r}{\beta}>d>\frac{r}{1+\beta}$ and  $c,\varepsilon>0$ sufficiently small, the deterministic model is not permanent, but the stochastic counterpart is stochastically persistent. This difference stems from the deterministic model having an internal equilibrium for species $1$ and $2$ whose external Lyapunov exponent is negative i.e. $r/(1+\beta)-d<0$. However, the stochastic model has no ergodic invariant measure supporting species $1$ and $2$ and, consequently, doesn't have this negative external Lyapunov exponent.
\end{rmk}

\section{Discussion}\label{s:disc}

Due to the irreducibility assumption (Assumption~\ref{a.tight}) of the stochastic Kolmogorov systems considered here, our process $\BX$ has a finite number of ergodic invariant probability measures in any dimension. However, in dimension $\le 3$, we prove there are constraints on what types of configurations of ergodic measures are possible. Moreover, we show that, generically, these configurations can be identified by studying the average per-capita growth rates of the infinitesimally rare species, i.e.. the external Lyapunov exponents $\lambda_i(\mu)$  that we have shown to be generically non-zero.

We find there are three basic types of asymptotic behavior. First, the Kolmogorov process $\BX$ may be stochastically persistent which corresponds to all the species persisting. Specifically, there is a unique ergodic measure $\mu$ supporting all the species. This ergodic measure characterizes (with probability one), the asymptotic, statistical behavior of $\BX$ for all strictly positive initial conditions $\BX(0)\gg 0$. In particular, for any continuous bounded function $h$ (i.e. an observable for the system), the temporal averages $\frac{1}{t}\int_0^t h(\BX(t))dt$ converge (with probability one) to the spatial average $\int h(\bx)\mu(\bx)$. Verifying stochastic persistence using the external Lyapunov exponents reduces to a simple procedure. First, for any ergodic measure $\mu$ supporting two or fewer species (i.e. $|I_\mu||\le 2$), there needs to be at least one species with a positive per-capita growth rate i.e. $\max_i \lambda_i(\mu)>0$. Second, if there is no rock-paper-scissor intransitivity between the species, then $\BX$ is stochastically persistent. Alternatively, if there is a rock-paper-scissor intransitivity, persistence requires that the sum of the product of the positive external Lyapunov exponents and the product of the negative Lyapunov external exponents is positive, where the products are taken over the single species ergodic measures.

The second and third form of asymptotic behaviors occur when the system is not stochastically persistent. In these cases, the process $\BX$ converges with probability one to the boundary of the three-dimensional, non-negative orthant. However, this convergence can take on two forms. The first form of extinction corresponds to ergodic measures $\mu$ that are attractors on the boundary of the orthant. An attractor is an ergodic measure $\mu$ such that $I_\mu\subsetneq \{1,\dots,n\}$ and $\max_{i\notin I_\mu}\lambda_i(\mu)<0$, i.e., the measure $\mu$ only supports a subset of the species and all its external Lyapunov exponents are negative. There can exist at most a finite number of these ergodic attractors, say $\mu^1,\dots,\mu^k$ (see Figure~\ref{fig:one}). The only constraint on these ergodic attractors is that a pair of them can not correspond to a nested pair of species i.e. $I_{\mu^i}$ is never a subset of $I_{\mu^j}$ for $i\neq j$. When these ergodic attractors exist and all species are initially present, the process converges with probability one to one of these attractors, and there is a strictly positive probability that it converges to any of the $k$ ergodic attractors. The second form of extinction corresponds to an attractor rock-paper-scissor dynamic on the boundary of the non-negative orthant. In this case, the asymptotic statistical behavior of $\BX$ is (with probability one) determined by convex combinations of the single species ergodic measures.

For higher dimensions, we conjecture there is a similar classification of the behaviors of $\BX$. In the simplest setting, when one looks at Lotka-Volterra food chains and each species only interacts with its immediate trophic neighbors the classification has been completed in \cite{HN17,HN17b}. The classification for general Kolmogorov systems will have to deal with higher dimensional analogs of the rock-paper-scissors intransitives. As already explored in deterministic models, these higher dimensional intransitivities may involve complex networks of transitions between subcommunities due to single or multiple species invasions~\citep{hofbauer1994,brannath1994,krupa1997,schreiber1998stabilizing,schreiber2004simple,vandermeer2011intransitive}. For example, \citet{schreiber1998stabilizing} illustrates that for a community of $n$ founder controlled prey species and $n$ specialist predators, the predator-prey pairs get displaced by the invasion of any other prey species which then facilitates the establishment of the predator. This leads to a high dimensional heteroclinic cycle. Despite these complexities, one might conjecture that one could extend the rock-paper-scissor extinction outcome to the existence of a finite number of ergodic measures such that with positive probability, the asymptotic behavior is determined by non-trivial convex combinations of these ergodic measures. Moreover, in higher dimensions, one would have to allow for the possibility that ergodic attractors and these non-ergodic, intransitive attractors can occur simultaneously to govern the extinction dynamics. Here, we have verified a key step for such a classification in higher dimensions by showing that the external Lyapunov exponents are, generically, non-zero.

Another important corollary of our work is with respect to \textit{modern coexistence theory} \citep{C00,ellner2019} -- this is fundamental framework that is widely used by theoretical ecologists to study the mechanisms underlying the coexistence of species. This theory is based entirely on using external Lyapunov exponents, also called \textit{invasion growth rates}. Our work shows for a general class of SDE models that the external Lyapunov exponents fully describe the long term behavior of the system and, thereby,  justifies rigorously the main premise of modern coexistence theory for these models.

\textbf{Acknowledgments:} The authors acknowledge
support from the NSF through the grants DMS-1853463 for
Alexandru Hening, DMS-1853467 for Dang Nguyen, and DMS-1716803 for Sebastian Schreiber.
\bibliographystyle{agsm}
\bibliography{LV}

\end{document}